\documentclass[12pt,a4paper]{article}

\usepackage{latexsym, amssymb, bbm, mathrsfs}
\usepackage{amsmath,amsfonts,amsthm}
\usepackage{a4wide}
\usepackage{times}
\usepackage[shortlabels]{enumitem}
\usepackage{url}
\usepackage{hyperref}

\newcommand{\Tr}{\mathop{\mathrm{Tr}}\nolimits}
\newcommand{\tr}{\mathop{\mathrm{Tr}}\nolimits}
\newcommand{\E}{\mathbb{E}}
\newcommand{\EE}{\mathbb{E}}
\newcommand{\PP}{\mathbb{P}}
\newcommand{\prob}{\mathbb{P}}
\newcommand{\e}{\mathrm{e}}
\newcommand{\RR}{\mathbb{R}}
\newcommand{\R}{\mathbb{R}}

\newcommand{\cov}{\mathop{\mathrm{Cov}}\nolimits}
\newcommand{\var}{\mathop{\mathrm{Var}}\nolimits}
\newcommand{\id}{\mathrm{Id}}
\newcommand{\eps}{\varepsilon}
\renewcommand{\epsilon}{\varepsilon}
\newcommand{\vphi}{\varphi}
\renewcommand{\phi}{\varphi}
\newcommand{\cC}{\mathcal C}

\renewcommand{\tilde}{\widetilde}
\newtheorem{theorem}{Theorem}[section]
\newtheorem{definition}[theorem]{Definition}

\newtheorem{lemma}[theorem]{Lemma}

\newtheorem{proposition}[theorem]{Proposition}
\newtheorem{corollary}[theorem]{Corollary}
\theoremstyle{definition}
\newtheorem{remark}[theorem]{Remark}

\newcommand{\cG}{\mathcal{G}}
\newcommand{\cH}{\mathcal{H}}
\newcommand{\cX}{\mathcal{X}}

\begin{document}

\title{Thin-shell bounds via parallel coupling}
\author{Boaz Klartag and Joseph Lehec}
\maketitle
\abstract{We prove that for any log-concave random vector $X$ 
in $\mathbb R^n$ with mean zero and identity covariance,
\[
\var ( |X| ) \leq \mathbb E (|X| - \sqrt{n})^2 \leq C 
\] 
where $C > 0$ is a universal constant, and 
$|X|$ denotes the Euclidean norm of the random vector $X$. 
Thus, most of the mass of $X$ is concentrated in a 
thin spherical shell, whose width is only
$C / \sqrt{n}$ times its radius. This confirms the thin-shell conjecture in high dimensional convex geometry.
Our method relies on the construction of a certain coupling between
log-affine perturbations of the law of $X$ related to Eldan's stochastic localization and to the theory of non-linear filtering.
Another ingredient is a recent breakthrough technique by Guan that was previously used in our proof of Bourgain's slicing conjecture, which is known to be implied by the thin-shell conjecture.
}

\section{Introduction}

A probability density $\rho$ in $\RR^n$ is log-concave if its support $K = \{x \in \RR^n \, ; \, \rho(x) > 0 \}$ is a convex
set, and $\log \rho$ is a concave function on $K$. 
A probability measure $\mu$ on $\RR^n$ is log-concave if it is absolutely-continuous with a log-concave density, or more generally, if it is supported in an affine subspace of $\RR^n$ and has a log-concave density in that subspace.
For example, the uniform probability measure
on any convex body in $\RR^n$ is log-concave, as well as all  Gaussian measures. The class of log-concave probability
measures is closed under convolutions, weak limits and push-forwards under linear maps, as follows from the Pr\'ekopa-Leindler
inequality (e.g. \cite[Theorem 1.2.3]{BGVV}).

\medskip
A log-concave probability measure has moments of all orders (e.g. \cite[Lemma 2.2.1]{BGVV}).
The covariance matrix of the log-concave probability measure $\mu$ is the matrix $\cov(\mu) = (\cov_{ij}(\mu))_{i,j=1,\ldots,n} \in \RR^{n \times n}$ where
$$ \cov_{ij}(\mu) = \int_{\RR^n} x_i x_j \, d \mu(x) - \int_{\RR^n} x_i \, d \mu(x) \cdot \int_{\RR^n} x_j \, d \mu(x). $$
The barycenter of $\mu$ is the vector $ \int_{\RR^n} x \, d \mu(x) \in \RR^n$. The probability measure $\mu$ is {\it centered} when its barycenter lies at the origin,
and it is  {\it isotropic} if it is centered and $$ \cov(\mu) = \id. $$
For a random vector $X$ in $\RR^n$ with law $\mu$ we denote $\cov(X) = \cov(\mu)$. 
We say that $X$ is log-concave (respectively, isotropic) if its law $\mu$ is log-concave (respectively, isotropic).
It is well-known that for any random vector $X$  with finite second moments whose support affinely spans $\RR^n$,
 there exists an affine map $T: \RR^n \rightarrow \RR^n$ such that $T(X)$ is isotopic (see, e.g. \cite[Section 2.3]{BGVV}). Our main result is the following:

\begin{theorem}\label{thm1}
Let $X$ be an isotropic, log-concave random vector in $\RR^n$. Then,
\begin{equation} \var(|X|^2) = \EE \left( |X|^2 - n \right)^2 \leq C n, \label{eq_1516} \end{equation}
where $C > 0$ is a universal constant.
\end{theorem}
 Theorem \ref{thm1} is tight, up to the value of the universal constant. Indeed, if $X$ is a standard
 Gaussian random vector in $\RR^n$ or if $X$ is distributed uniformly in
 the cube $[-\sqrt{3}, \sqrt{3}]^n \subseteq \RR^n$, then $X$ is isotropic and log-concave with
\begin{equation}\label{eq_zzz}
\var(|X|^2) = C n, 
\end{equation} 
where $C = 2$ in the Gaussian case and $C = 4/5$ in the case of the cube.
In fact, whenever $X = (X_1,\dotsc,X_n)$ is a random vector having 
i.i.d. coordinates with expectation $0$, variance $1$ and finite fourth 
moment, then $X$ is isotropic and (\ref{eq_zzz}) holds with $C = \E X_1^4 -1$. 
In a sense, Theorem \ref{thm1} shows that as far as the variance of the norm 
is concerned, log-concave random vectors behave just as well as 
random vectors having independent coordinates.
  
The inequality
$$ \EE \left( |X| - \sqrt{n} \right)^2 \leq \EE \left( |X| - \sqrt{n} \right)^2 \frac{\left( |X| + \sqrt{n} \right)^2}{n}
= \frac{1}{n} \cdot \EE (|X|^2 - n)^2  $$
shows that Theorem \ref{thm1} admits the following consequence. 
\begin{corollary} \label{cor_zzz}
If $X$ is log-concave and isotropic then 
\begin{equation}\label{abs}
\var ( \vert X\vert ) \leq \E ( \vert X\vert - \sqrt n)^2 \leq C , 
\end{equation} 
where $C$ is a universal constant. 
\end{corollary} 
In other words, the standard deviation of the norm of an isotropic 
log-concave random vector is of constant order, which is roughly 
$1/\sqrt n$ times its average. 
Moreover, one can push this concentration property a little 
further. Indeed, reverse H\"older inequalities for polynomials of a random vector distributed uniformly in a convex body were established by Bourgain \cite{bourgain}.
By using the version of these inequalities from Nazarov, Sodin and Volberg \cite{nsv} together with Theorem \ref{thm1} we obtain the following: 
\begin{corollary}  
For any isotropic, log-concave random vector $X$ in $\RR^n$ and any $t > 0$,
\begin{equation}
\PP \left( \big| \hspace{0.3pt} |X| - \sqrt{n}  \hspace{0.5pt} \big| \geq t \right) \leq \PP \left( \left| \frac{|X|^2 - n}{\sqrt{n}} \right| \geq t \right) \leq C \exp(-c \sqrt{t}).
\label{eq_1737_}
\end{equation}
\end{corollary}
Inequality (\ref{eq_1737_}) is known to be suboptimal for large values of $t$
(e.g., Paouris \cite{paouris} or \cite[Section 8.2]{KLbull}). Nevertheless,  it is a {\it thin-shell bound}, since for $1 \ll t \ll \sqrt{n}$ inequality (\ref{eq_1737_}) implies that
with overwhelming probability, the random vector $X$ belongs to the thin spherical shell $\{ x \in \RR^n \, ; \, \sqrt{n} -t \leq |x| \leq \sqrt{n} + t\}$, whose width $t$ is much smaller than its radius $\sqrt{n}$.

\medskip
The equivalence between thin-shell bounds and the Gaussian approximation property
of typical marginal distributions goes back to Sudakov \cite{sudakov} and to Diaconis and Freedman \cite{diaconis_freedman}.
See e.g. Bobkov, Chistyakov and G\"otze \cite{BCG1}
or \cite{K_euro} for more information; in particular, a thin-shell bound lies at the heart
of the proof of the central limit theorem for convex bodies.

\medskip Under convexity assumptions, thin-shell bounds in the spirit of  (\ref{eq_1737_}) were  conjectured
by Anttila, Ball and Perissinaki \cite{ABP} in the context of the central limit problem for convex bodies.
In the case where $X$ is distributed uniformly in a convex body,
the precise form of Theorem \ref{thm1}  was posed as an open problem by Bobkov and Koldobsky \cite{BobKol},
who also observed that an affirmative answer would follow from the Kannan-Lov\'asz-Simonovits (KLS) conjecture.

\medskip The thin-shell conjecture (i.e., the statement of Theorem \ref{thm1}) is sometimes referred to as the {\it variance conjecture}
and it is related  to Bourgain's slicing problem.
In fact, Eldan and Klartag \cite{EK} used the logarithmic Laplace transform and the Bourgain-Milman inequality \cite{bourgain_milman}
in order to show that the thin-shell conjecture implies an affirmative
answer to Bourgain's slicing problem. Thus, for quite some time, the thin-shell conjecture was considered ``harder'' than the slicing
problem but ``easier'' than the KLS conjecture. Bourgain's slicing problem was resolved in the affirmative in \cite{KL_slicing}
by using a recent bound by Guan \cite{guan}. Guan's technique is also a crucial ingredient in the proof of Theorem \ref{thm1} presented below.

\medskip In the case where the random vector $X$ is distributed uniformly on a suitably-scaled $\ell_p^n$-ball, the conclusion of
Theorem \ref{thm1} follows from the work of Ball and Perissinaki \cite{ball_perissinaki}. In the case
where $X$ is distributed uniformly in a convex body $K \subseteq \RR^n$ with coordinate symmetries
(i.e., $(x_1,\ldots,x_n) \in K \iff (\pm x_1,\ldots,\pm x_n) \in K$), the conclusion of Theorem \ref{thm1}
was proven in \cite{K_ptrf}. The thin-shell conjecture was proven under symmetry assumptions of various types by Barthe and Cordero-Erausquin \cite{BCE},
and for Schatten class bodies by Radke and Vritsiou \cite{RV} and 
Dadoun, Fradelizi, Gu\'edon and Zitt \cite{DFGZ}. The stronger KLS conjecture was established for Orlicz balls by Kolesnikov and Milman 
\cite{KolesMil} and  Barthe and Wolff \cite{BW}. 

\medskip In the general case, the first non-trivial upper bound for the left-hand
side of (\ref{abs}) was given in the proof of the central limit theorem for convex sets in \cite{K_clt}, which was influenced by the earlier work of Paouris \cite{paouris}.
The bound obtained was that for an isotropic, log-concave random vector $X$ in $\RR^n$,
$$ \EE \left(|X| - \sqrt{n} \right)^2 \leq \sigma_n^2 $$
with $\sigma_n \leq C \sqrt{n} / \log n$. This bound was improved
to $\sigma_n \leq C n^{2/5 + o(1)}$ in \cite{K_poly}, to $\sigma_n \leq C n^{3/8}$ in Fleury \cite{fleury}
and to $\sigma_n \leq C n^{1/3}$ in Gu\'edon and Milman \cite{guedon_milman}. Roughly speaking, the proofs of these bounds
relied on concentration of measure on the high-dimensional sphere.
 Eldan's stochastic localization was then used by Lee and Vempala \cite{LV_focs} in order to show that in fact $\sigma_n \leq C n^{1/4}$.
 Thanks to Eldan and Klartag \cite{EK}, this yielded another proof
 of the $n^{1/4}$-bound for Bourgain's slicing problem, which was the state of the art at the time, and was speculated by some to be optimal. However, the methods of Lee and Vempala were extended  in a breakthrough work by Chen \cite{chen} who came up with a clever growth regularity estimate and proved  the bound
 $$ \sigma_n \leq C \exp \left( ( \log n)^{1/2 + o(1)} \right) = n^{o(1)}. $$
 This was improved to $\sigma_n \leq C \log^4 n$ in \cite{KL} by combining Chen's work with spectral analysis, and then
 to $\sigma_n \leq C \log^{2.23} n$ in Jambulapati, Lee and Vempala
 \cite{JLV} by refining the method from \cite{KL}. The bound $\sigma_n \leq C \sqrt{\log n}$ was then obtained in \cite{K_root} by replacing the use of growth regularity estimates with an improved Lichnerowicz inequality. This inequality was then used in an extremely intricate bootstrap analysis in Guan \cite{guan}, which we discuss in great detail below, for proving $\sigma_n \leq C \log \log n$. Note that the bound $\sigma_n \leq C$ follows from our main result, see Corollary \ref{cor_zzz}. 
 
\medskip Our proof of Theorem \ref{thm1} employs an idea from the proof of the thin-shell conjecture under coordinate symmetries in \cite{K_ptrf}, which is 
also the starting point of the most recent estimates in the 
general case \cite{KL,JLV,guan}. 
Let $\mu$ be a log-concave probability measure in $\RR^n$. As in \cite{barthe_klartag}, we define the space $H^1(\mu)$ to be the collection of all functions $f \in L^2(\mu)$ with weak partial derivatives in $L^2(\mu)$,  equipped with the norm
$$ \| f \|_{H^1(\mu)} = \sqrt{ \int_{\RR^n} |f|^2 \, d \mu + \int_{\RR^n} |\nabla f|^2 \, d \mu}. $$
In particular, the space $H^1(\mu)$ contains all locally-Lipschitz functions $f \in L^2(\mu)$ with $\partial_1 f, \ldots, \partial_n f \in L^2(\mu)$.
It is proven in Barthe and Klartag \cite{barthe_klartag} that the space $\mathcal C_c^{\infty}(\RR^n)$ of smooth, compactly-supported functions in $\RR^n$, is a dense subspace of the Hilbert space $H^1(\mu)$.
For a function $f \in L^2(\mu)$ with $\int f d \mu = 0$ we define
\begin{align}
	\nonumber
	\| f \|_{H^{-1}(\mu)}
& = 	\sup \left \{ \int_{\RR^n} f g \, d \mu \, ; \, g \in H^1(\mu), \ \int_{\RR^n} |\nabla g|^2 \, d \mu \leq 1 \right \}
	\\ &= \sup \left \{ \int_{\RR^n} f g \, d \mu \, ; \, g \in \mathcal C_c^{\infty}(\RR^n), \ \int_{\RR^n} |\nabla g|^2 \, d \mu \leq 1 \right \}.
	\label{eq_1528}
\end{align}
The $H^{-1}(\mu)$-norm is related to infinitesimal Optimal Transport, see e.g. Villani \cite[Section 7.6]{Vil} or the Appendix of \cite{K_ptrf}.
It was shown
in \cite{barthe_klartag} and \cite{K_ptrf}
by using the {\it Bochner formula} that
 for any smooth function $f \in H^1(\mu)$ with $\int f d \mu = 0$ and $\int \nabla f d \mu = 0$,
\begin{equation}
\| f \|_{L^2(\mu)}^2 \leq \| \nabla f \|_{H^{-1}(\mu)}^2 :=
\sum_{i=1}^n \| \partial_i f \|_{H^{-1}(\mu)}^2. \label{eq_1020}
\end{equation}
Let us apply (\ref{eq_1020}) in the particular case where the log-concave probability measure $\mu$ is isotropic
and where $f(x) = |x|^2 - n$. In this case, $\int f d \mu = 0, \nabla f(x) = 2x$ and $\int \nabla f d \mu = 0$.
 It therefore follows from (\ref{eq_1020}) then if $X$ is a random vector with law $\mu$, then
\begin{equation}
\var(|X|^2) = \EE (|X|^2 - n)^2 \leq 4 \sum_{i=1}^n \| x_i \|^2_{H^{-1}(\mu)}. \label{eq_1028}
\end{equation}
Consequently, as in \cite{K_ptrf}, Theorem \ref{thm1} would follow from (\ref{eq_1028}) once we prove the following:

\begin{theorem} Let $\mu$ be an isotropic, log-concave probability measure in $\RR^n$. Then,
	$$ \sum_{i=1}^n \| x_i \|^2_{H^{-1}(\mu)} \leq C n, $$
where $C > 0$ is a universal constant.	\label{thm2}
\end{theorem}

In order to prove Theorem \ref{thm2}, we use a different approach than 
the spectral method introduced in \cite{KL}, which is also used in \cite{JLV,guan}. Instead, we first consider, in Section \ref{sec2}, 
the family of {\it exponential tilts} or 
{\it log-affine perturbations} of the measure $\mu$, 
and we construct certain couplings between these tilts, 
which are related to Eldan's stochastic localization. 
In Section \ref{sec3} we use the optimal transport interpretation
of the $H^{-1}(\mu)$-norm as well as the log-concavity assumption, and show that these couplings allow to bound the $H^{-1}(\mu)$-norm
by the growth of the covariance process $(A_t)_{t \geq 0}$ of the stochastic localization, see Corollary \ref{cor_311}. These first two steps are the main 
novelty of the paper. 
In Section~\ref{sec4} we analyze the eigenvalues of the process $(A_t)$ by using a variant of Guan's technique. This allows us to prove in Section
 \ref{sec6} a new estimate for the covariance of the stochastic localization, Theorem \ref{thm_guantype}, 
 which, together with Corollary \ref{cor_311}, yield Theorem \ref{thm2}. 
Section~\ref{digress} is not directly relevant to the proof of the thin-shell conjecture; it is a digression on a natural stochastic process of martingale diffeomorphisms associated with the measure~$\mu$, which stems from our construction.

\medskip Our notation is fairly standard. We write $x \cdot y = \langle x, y \rangle = \sum_i x_i y_i$ for the
scalar product between $x,y \in \RR^n$, and $|x| = \sqrt{\langle x, x \rangle}$ is the Euclidean norm.
For a matrix $A \in \RR^{n \times n}$ we write $A^*$ for its transpose.
For two symmetric matrices $A, B \in \RR^{n \times n}$ we write $A \leq B$ if $B-A$ is positive
semi definite.
For $x \in \RR^n$ we write $$ x \otimes x = (x_i x_j)_{i,j=1,\ldots,n} \in \RR^{n \times n}. $$
A smooth function or a diffeomorphism are  $\mathcal C^{\infty}$-smooth, unless stated otherwise.
For a smooth map $F: \RR^n \rightarrow \RR^n$ we write $F'(x) \in \RR^{n \times n}$ for the derivative matrix of $F$ at the point $x \in \RR^n$. That is,
$\partial_v F(x) = F'(x) v$, where $\partial_v F$ is the directional derivative of $F$ in direction $v \in \RR^n$.
For a smooth function $f: \RR^n \rightarrow \RR$ we write $\nabla^2 f(x) \in \RR^{n \times n}$
for its Hessian matrix at the point $x \in \RR^n$. A map $f: \RR^n \rightarrow \RR^n$
is expanding if $|f(x) - f(y)| \geq |x-y|$ for all $x,y \in \RR^n$, and it is $L$-Lipschitz
if $$ |f(x) - f(y)| \leq L \cdot |x-y| \qquad \qquad \qquad \textrm{for all} \ x,y \in \RR^n. $$
The support of a Borel measure $\mu$ on $\RR^n$ is the closed set whose complement is the union of all open sets of zero $\mu$-measure.
We write $\overline{A}$ for the closure of the set $A \subseteq \RR^n$.
We write $C, c, C', \tilde{c}, \bar{C}$ etc. to denote various positive universal constants whose value may change from one line to the next.

\medskip {\it Acknowledgements.} We are grateful to Qingyang Guan for valuable discussions, and to Ramon van Handel and Ofer Zeitouni
for information on the theory of non-linear filtering. BK was supported by a grant from the Israel Science Foundation (ISF).

\section{Coupling of tilts}
\label{sec2}

Let $\mu$ be a compactly-supported probability measure whose support
affinely spans $\RR^n$. For $t \geq 0$ and $\theta \in \RR^n$ we consider
the {\it logarithmic Laplace transform}
$$ \Lambda_t(\theta) = \log \int_{\RR^n} \exp \left( \langle \theta , x \rangle - \frac t2 |x|^2 \right) \, d \mu (x) . $$
The logarithmic Laplace transform $\Lambda_t$ is a smooth, convex function in $\RR^n$, and its derivatives are expressed below via the probability measure
$\mu_{t,\theta}$ defined by
\begin{equation}
\label{eq_1045}
\frac{d \mu_{t,\theta}}{d\mu} (x) = \exp \left( \langle x,\theta\rangle -\frac t2 \vert x\vert^2 - \Lambda_t ( \theta ) \right).
\end{equation}
By the definition of $\Lambda_t$, the measure $\mu_{t,\theta}$
is indeed a probability measure. We abbreviate $$ \mu_{\theta} := \mu_{0,\theta}. $$ The family of measures $(\mu_\theta)_{\theta \in \RR^n}$ 
is the family of  \emph{log-affine perturbations} or 
\emph{exponential tilts} of the measure $\mu$. These measures were used in a similar context already in \cite{K_quarter}.
In this section we construct couplings between different tilts
of the measure $\mu$.
Our construction draws heavily from the theory of non-linear filtering \cite{chigansky} and Eldan's stochastic localization~\cite{eldan1}.

%
%
\medskip We begin by differentiating $\Lambda_t$ under the integral sign.
We see that $\nabla \Lambda_t(\theta)$ equals the barycenter of $\mu_{t,\theta}$,
which we shall denote by 
\begin{equation}
a(t,\theta) = \nabla \Lambda_t(\theta) = \int_{\RR^n} x \, d \mu_{t, \theta} (x)  \in \RR^n .  \label{eq_1329}
\end{equation}
Similarly, the second derivative coincides with the covariance
matrix, denoted by 
\begin{equation}
	A(t, \theta) = \nabla^2 \Lambda_t(\theta) = \int_{\RR^n}
	x \otimes x \, d \mu_{t,\theta} (x) - a(t, \theta) \otimes a(t, \theta) \in \RR^{n \times n}.	\label{eq_1122}
\end{equation}
The covariance matrix $A(t, \theta) \in \RR^{n \times n}$ is symmetric and positive-definite. Since $\mu$ is compactly-supported and the measures $\mu_{t,\theta}$ have the same support
as $\mu$, there exists a constant $C_\mu > 0$ depending only on $\mu$ such that for any $t\geq 0$ and $\theta \in \R^n$,
\begin{equation}\label{eq_1123_}
  |a(t, \theta)| = |\nabla \Lambda_t(\theta)| \leq C_{\mu} \end{equation}
while
\begin{equation} 
0 \leq A(t,\theta) = \nabla^2 \Lambda_t (\theta) \leq C_\mu \cdot \id.
\label{eq_727} \end{equation}
We conclude that $a(t,\cdot)= \nabla \Lambda_t(\cdot)$ is a $C_{\mu}$-Lipschitz map, i.e., 
\begin{equation}\label{eq_1123}
\vert a(t,\theta_1) - a(t,\theta_2) \vert \leq C_\mu \vert \theta_1 -\theta_2\vert , \qquad  t \geq 0 , \, \theta_1,\theta_2\in \R^n.
\end{equation}
Throughout this paper, we write $\mathcal C ( [0,\infty) , \R^n )$
for the space of all continuous paths $(w_t)_{t \geq 0}$ in $\R^n$. 
We equip this space with the
topology of uniform convergence on compact intervals, and with the
corresponding Borel $\sigma$-algebra.

\begin{lemma} \label{lem_1800}
Fix $w = (w_t)_{t \geq 0} \in \cC([0, \infty), \RR^n)$. Then
for any $x\in \R^n$ there exists a unique solution
$(\theta_t)_{t \geq 0}$ to the integral equation
\begin{equation}
\theta_t = x + w_t + \int_0^t a(s,  \theta_s) ds ,
\qquad t \geq 0.
\label{eq_1801} \end{equation}
The solution $\theta_t = \theta_t (x)$ is continuous in $(t,x) \in [0, \infty) \times \RR^n$
and is smooth in $x \in \RR^n$ for any fixed $t \geq 0$. 

\medskip Moreover, the derivative $M_t (x) = \theta_t'(x) \in \RR^{n \times n}$
of the smooth map $\theta_t: \RR^n \rightarrow \RR^n$ satisfies the following: the matrix $M_t(x)$ is continuous
in $(t,x) \in [0, \infty) \times \RR^n$ and $\mathcal C^1$-smooth in $t > 0$, and it is the unique solution of the
linear differential equation
\begin{equation} \label{eq_1802}
\begin{cases}
M_0 (x) = \id \\
\frac d{dt} M_t(x) = A(t,\theta_t(x)) M_t(x),  & \phantom{abc}t\geq 0.
\end{cases}
\end{equation}
\end{lemma}
\begin{proof} 
Fix a continuous path $w = (w_t)_{t \geq 0}$.
Observe that $(\theta_t)$ satisfies (\ref{eq_1801})
if and only if the path $(y_t)$ given by $y_t = \theta_t - w_t$  satisfies
\begin{equation} \label{eq_1803}
y_t = x + \int_0^t a(s, w_s + y_s) \, ds , \qquad t >0 .
\end{equation}
The vector $a(t,x) \in \RR^n$ depends continuously on $(t,x) \in [0, \infty) \times \RR^n$ while $(w_t)$
is a continuous path. Hence the map $(t,x) \mapsto a(t,w_t +x) $ is
 continuous as well. An equivalent formulation
of the integral equation (\ref{eq_1803}) is that the path $(y_t)$ needs to solve the ordinary differential
equation
\begin{equation} \label{eq_1804}
\begin{cases}
y_0 = x \\
\frac d{dt} y_t = a( t  ,w_t + y_t) , \qquad t > 0 .
\end{cases}
\end{equation}
From (\ref{eq_1123_}) and (\ref{eq_1123}) we know that 
$x \mapsto a(t,w_t+x)$ is bounded and Lipschitz continuous, uniformly in $t \in [0, \infty)$.
By the Cauchy-Lipschitz theorem, which is also called the Picard-Lindel\"of theorem, equation \eqref{eq_1804} has a unique solution
(see e.g. Hartman \cite[Theorem 1.1]{hartman}). This shows that (\ref{eq_1801}) has a unique solution.

\medskip Moreover, for any fixed $t \geq 0$, the map $x\mapsto a(t,w_t+x)$ is smooth. 
A slightly more advanced version of the Picard-Lindel\"of theorem from \cite[Chapter V]{hartman}
then shows that $y_t(x)$ is continuous in $(t,x) \in [0, \infty) \times \RR^n$ and
$\mathcal C^{\infty}$-smooth in $x \in \RR^n$ for any fixed $t \geq 0$. 
Let $\theta_t = \theta_t (x)$ be the unique solution of (\ref{eq_1801}) and consider the unique solution $y_t = y_t(x)$ of (\ref{eq_1804}). Recalling that
$$
\theta_t (x) = w_t + y_t ( x ) , \qquad t \geq 0 ,\; x \in \R^n,
$$ 
we conclude that $\theta_t (x)$ is continuous in $(t,x) \in [0, \infty) \times \RR^n$
and $\mathcal C^{\infty}$-smooth in $x \in \RR^n$ for any fixed $t \geq 0$. 

\medskip Furthermore, according to
 \cite[Theorem 3.1]{hartman}, the spatial derivative $y_t'$ 
of $y_t$ is $\mathcal C^1$-smooth in $t$, and we can differentiate
(\ref{eq_1804}) with respect to $x$. The derivative $y_t'$ is the unique solution to the ordinary 
differential equation obtained by differentiating (\ref{eq_1804}) with respect to $x$,
and it is jointly continuous in $(t,x) \in [0, \infty) \times \RR^n$. Thus equation (\ref{eq_1802})
holds true and $$ M_t(x) = \theta_t'(x) = y_t'(x) \in \RR^{n \times n} $$ is
$\mathcal C^1$-smooth in $t > 0$ and continuous in $(t,x) \in [0, \infty) \times \RR^n$.
\end{proof}
\begin{definition} We denote by $G = (G_{t,w})_{t \geq 0}$ the flow associated with the integral equation (\ref{eq_1801}). That is, 
for any $t \geq 0, x\in \RR^n$ and $w \in \cC([0, \infty), \RR^n)$, the vector $G_{t,w}(x) \in \RR^n$  is the value at time $t$ of the unique solution $\theta = (\theta_t)_{t \geq 0}$ of (\ref{eq_1801}).
\label{def_1607}
\end{definition}
Next we investigate the dependence on $w$ of the flow $G$.
We let $(\mathcal F_t)_{t \geq 0}$ be the natural filtration of the coordinate
process on $\mathcal C ([0,\infty) , \R^n )$. In other words,
$\mathcal F_t$ is the smallest $\sigma$-algebra with respect to which the
map $w \mapsto w_s$ is measurable for any $s \in [0,t]$. It is
well-known that the $\sigma$-algebra generated by
$\bigcup_{t>0} \mathcal F_t$ coincides with the Borel $\sigma$-algebra
of $\mathcal C ([0,\infty), \R^n)$.
\begin{lemma} \label{lem_1300}
The map $(x,w ) \mapsto (G_ {t,w} (x))_{t\geq 0} \in \mathcal C ([0,\infty) , \R^n )$
is continuous. Moreover, for every fixed $t\geq 0$ and $x\in \R^n$,
the map $w \mapsto G_{t,w} (x)$ is $\mathcal F_t$-measurable.
\end{lemma}
\begin{proof} Since we were not able to find this result in the literature
on ordinary differential equations, we provide an ad-hoc argument.
Fix $x,\tilde x \in \R^n$ and $w,\tilde w \in \mathcal C ([0,\infty) ,\R^n)$,
and let $\theta_t = G_{t,w} (x)$ and $\tilde \theta_t = G_{t,\tilde w} (\tilde x)$. We use (\ref{eq_1801}), the triangle inequality
and the fact that $x\mapsto a(t,x)$ is $C_\mu$-Lipschitz to obtain
\[
\begin{split}
\vert \theta_t -\tilde \theta_t \vert
\leq \vert x-\tilde x \vert +
\vert  w_t - \tilde w_t \vert +
C_\mu \int_0^t  \vert \theta_s -\tilde \theta_s \vert \, ds .
\end{split}
\]
Solving this differential inequality (Gronwall's lemma) we get
\[
\vert \theta_t - \tilde \theta_t \vert
\leq \e^{C_\mu t} \vert x-\tilde x\vert
+ \vert w_t - \tilde w_t \vert + C_\mu \int_{0}^t \e^{ C_\mu (t-s) } \vert w_s - \tilde w_s \vert \, ds .
\]
This implies that for all $t \geq 0, x,\tilde x \in \R^n$ and $w,\tilde w \in \mathcal C ([0,\infty), \R^n)$,
\begin{equation}\label{eq_1301}
\vert G_{t,w} (x) - G_{t,\tilde w} (\tilde x) \vert \leq \e^{C_\mu t} \left( \vert x-\tilde x\vert
+ \sup_{s\in [0,t]} \{ \vert w_s - \tilde w_s \vert \} \right). 
\end{equation}
This inequality clearly yields the first statement of the lemma.
Moreover, it also
implies that if $w_s = \tilde w_s$ for all $s\leq t$,
then $G_{t, w}(x) = G_{t, \tilde w}(x)$. This
is a reformulation of the fact that $w \mapsto G_{t,w} (x)$ is
$\mathcal F_t$-measurable.
\end{proof}

In the course of the proof of Lemma \ref{lem_1300},
and more specifically in 
equation (\ref{eq_1301}), we actually proved the following:
\begin{lemma}\label{lem_1400}
For  $t\geq0$ and  $w\in \mathcal C ([0,\infty), \R^n)$, the map
$G_{t,w}: \RR^n \rightarrow \RR^n$ is $\e^{C_\mu t}$-Lipschitz.
\end{lemma}
Next we inject randomness into the construction. 
Let $(\Omega,\mathcal F,\prob)$ be a probability space
and let $B = (B_t)_{t \geq 0}$ be a standard Brownian motion in $\RR^n$ defined
on this probability space with $B_0 = 0$. We assume that the probability space 
is sufficiently large so that there exists a standard Gaussian random variable 
defined on this space which is independent of $(B_t)$. 

\medskip Note that $B \in \cC([0, \infty), \RR^n)$ almost surely. For $x\in \R^n$ consider the stochastic process
$(\theta_t^x)_{t \geq 0}$ given by $$ \theta_t^x = G_{t,B} (x), \qquad t \geq 0. $$
By Lemma~\ref{lem_1800} and Lemma~\ref{lem_1300}, it is a continuous stochastic process,
adapted to the natural filtration of the Brownian motion $(B_t)_{t \geq 0}$.
Equation (\ref{eq_1801}) can be interpreted as a stochastic
differential equation, rewritten as
\begin{equation} 
\theta_0^x =x , \quad d \theta_t^x = d B_t + a(t,\theta_t^x) \, dt, \qquad t \geq 0.
\label{eq_1025} \end{equation}
Although the existence and uniqueness of a solution to (\ref{eq_1025}) is
guaranteed by general results on stochastic differential
equations, this is not the approach we take here. Instead, the less
sophisticated \emph{pathwise approach} provided
by Lemma \ref{lem_1800} seems more convenient for
our purposes. The process $(\theta_t^x)$ has a
a rather explicit description, as we shall see next:
\begin{proposition}
\label{prop_1824}
Fix $x \in \RR^n$, and let $X$ be a random vector with law $\mu_{x}$ that is independent of the process $(B_t)_{t \geq 0}$.
Then the process $(G_{t,B} (x))_{t\geq 0}$
has the same law as the process
\begin{equation}
( x+ B_t + t X)_{t\geq 0} .
\label{eq_1929}
\end{equation}
\end{proposition}

\begin{proof}
In the case where $x=0$, this is proved e.g. in~\cite[Proposition 6.7]{KLbull},
using the Girsanov change of measure formula. That proof can easily
be adapted to the case of general $x \in \RR^n$, but we prefer to provide
here an alternative proof, relying on ideas from non-linear filtering theory.
For $t \geq 0$ denote
\begin{equation} \label{eq_defXt}
X_t  = t X + B_t.
\end{equation}
Let $(\mathcal G_t)_{t \geq 0}$ be the natural filtration of the process $(X_t)_{t \geq 0}$,
that is, $\mathcal G_t$ is the $\sigma$-algebra generated by the collection
of random variables $(X_s)_{0 \leq s\leq t}$. 
We think of $X_t$ (or rather $X_t/t$)
as a noisy observation of  $X$, and of the $\sigma$-algebra $\mathcal G_t$ as representing the total information available to the observer at time $t$.
A basic computation, going back to Cameron and Martin \cite{CM} in the 1940s,
and discussed in detail also in Chiganski \cite[Example 6.15]{chigansky} and in Klartag and Putterman \cite[Section 4]{KP} yields
\begin{equation}
\EE \left[ X \mid \mathcal G_t  \right]
= \EE \left[ X \mid X_t  \right] = a(t, x + X_t).
\label{eq_1136}
\end{equation}
Define
\begin{equation}  \label{eq_1709}
\tilde B_t = X_t - \int_0^t  \EE[ X | \mathcal G_s ] \, ds
 = X_t - \int_0^t a(s, x + X_s ) \, ds,
\end{equation}
and note that almost surely  $(\tilde{B}_t)_{t \geq 0} \in \cC([0,\infty), \RR^n)$.
By setting 
\[
\theta_t = x + X_t = x + tX + B_t,
\]
we may rewrite (\ref{eq_1709}) as
\begin{equation} \label{eq_1106} 
\theta_t
= x + \tilde B_t + \int_0^t a(s, \theta_s ) \, ds , \qquad \forall t \geq 0.
\end{equation}
From (\ref{eq_1106}) we see that 
\begin{equation} 
\theta_t = G_{ t , \tilde B} (x), \qquad \forall t \geq 0. 
\label{eq_1110} \end{equation}
Our goal is to prove that $( \theta_t)_{t \geq 0}$
has the same law as the process $(G_{t, B} (x))_{t \geq 0}$.
Thanks to (\ref{eq_1110}), this would follow once we prove that $(\tilde B_t)_{t \geq 0}$ coincides in law with $(B_t)_{t \geq 0}$. 

\medskip 
In other words, it suffices to prove that $(\tilde B_t)$ is a standard
Brownian motion. This is a basic result in non-linear filtering theory, in
which $(\tilde B_t)$ is called the \emph{innovation process} of $(X_t)$.
We provide the argument for completeness. Observe first that
\[
B_t-\tilde B_t = - t X + \int_0^t a(s, \theta_s ) \, ds
\]
is almost surely an absolutely-continuous function of $t$.
This already implies that
 $(B_t)$
and $(\tilde B_t)$ have the same quadratic covariation,
namely
\[
[ B]_t = [\tilde B]_t = t \cdot  \id , \quad t >0 .
\]
Recall that $(\tilde B_t)_{t \geq 0}$ is a continuous stochastic process with $\tilde{B}_0 = 0$. 
By L\'evy's characterization of the standard Brownian motion (e.g. \cite[Section 5.3.1]{legall}),
all that remains is to prove that $(\tilde B_t)$ is a martingale. 
We see from (\ref{eq_1709}) that $\tilde B_t$ is $\mathcal G_t$-measurable, and we need to prove that 
for fixed $0 \leq s \leq t$, 
\begin{equation}
\E [ \tilde B_t \mid \mathcal G_s ] = \tilde B_s.
\label{eq_1206}
\end{equation}
To this end, we recall (\ref{eq_defXt}) and (\ref{eq_1709}), and write
\begin{equation}\label{eq_ipo567}
\begin{split}
\E [ \tilde B_{t} \mid \mathcal G_{s} ]
& = \E [ B_{t} \mid \mathcal G_s ] + t \cdot \E [ X \mid \mathcal G_s ]
-\int_0^t  \E [ X \mid \mathcal G_{r\wedge s} ] \, dr \\
& = \E [ B_t \mid \mathcal G_s ] + s \cdot \E [ X \mid \mathcal G_s ]
- \int_0^s \E [ X \mid \mathcal G_r ] \, dr  ,
\end{split}
\end{equation}
where $r \wedge s = \min \{r,s \}$ and we used that for $r,s > 0$, 
\[
\E \left[ \E [ X \mid \mathcal G_{r} ] \mid \mathcal G_{s} \right]
= \E [ X \mid \mathcal G_{r \wedge s} ] .
\]
Observe that the random vector $B_t - B_s$ has mean zero and
is independent of $\mathcal G_s$, hence it also
has mean zero conditionally on $\mathcal G_s$. Consequently,
\[
\E [ B_{t} \mid \mathcal G_s ] = \E [ B_s \mid \mathcal G_s ].
\]
By substituting this back into~(\ref{eq_ipo567}), and using the fact that $X_s$ is
$\mathcal G_s$-measurable, we obtain
\[
\E [ \tilde B_t \mid \mathcal G_s ]
= \E [ X_s \mid \mathcal G_s ] - \int_0^s \E [ X \mid \mathcal G_r ] \, dr
=  X_s - \int_0^s \E [ X \mid \mathcal G_r ] \, dr  = \tilde B_s ,
\]
proving (\ref{eq_1206}). This completes the proof of the proposition.
\end{proof}
\begin{corollary} \label{cor_1000}
For any $x\in \R^n$, the random vector $G_{t,B}(x) / t$ converges
almost surely as $t$ tends to $+\infty$, and the limit has
law $\mu_x$.
\end{corollary}
\begin{proof}
By Proposition \ref{prop_1824} it suffices to show that
$(x + B_t + t X )/t$ converges almost surely and that the limit has law $\mu_x$. 
This simply follows from the fact that $B_t / t \to 0$ almost surely, and hence 
$(x + B_t + t X )/t \longrightarrow  X$ as $t \longrightarrow  \infty$, while $X$ has law $\mu_x$.
\end{proof}

Recall that a pair of random variables $X_1,X_2$ is a {\it coupling} 
of the probability measures $\nu_1, \nu_2$ if the two random variables are defined on the same 
probability space and if $X_i$ has law $\nu_i$ for $i=1,2$.
By using the same Brownian motion for different values of $x$, we
construct a coupling between  exponential tilts
of the measure $\mu$. More precisely, for every $x_1,x_2\in \R^n$,
\[
\lim_{t\to \infty} \frac{ G_{t,B} (x_1)}{t} \qquad \textrm{and} \qquad \lim_{t\to \infty} \frac{ G_{t,B} (x_2)}{t} 
\]
is a pair of random vectors in $\RR^n$ which provides a coupling of the measures $\mu_{x_1}$ and $\mu_{x_2}$.
This is called \emph{parallel coupling}, since
the infinitesimal Brownian steps of the two processes
$G_{t,B}(x_1)$ and $G_{t,B}(x_2)$ remain parallel. This
stands in contrast with the more sophisticated \emph{reflection coupling}
of Cranston and Kendall~\cite{LRog}, in which the Brownian
increments of the two processes mirror each other.

\section{A digression: martingale diffeomorphisms}
\label{digress}

Reader interested only in the solution of the
thin-shell problem may skip this section, 
in which we notice that the
above construction yields the existence of a certain stochastic process of a diffeomorphisms associated with the measure $\mu$.
Recall that $\mu$ is a compactly-supported probability measure whose support
affinely spans $\RR^n$. Write $$ K \subseteq \RR^n $$ for
the interior of the convex hull of the support of $\mu$.
The first observation is that the flow maps $(G_{t,w})_{t \geq 0}$ are diffeomorphisms of $\R^n$. 

\begin{proposition} \label{prop_1456}
For any $t \geq 0$ and any $w \in \cC([0, \infty), \RR^n)$, the map $G_{t,w}: \RR^n \rightarrow \RR^n$ is a diffeomorphism that is also an expanding map.
\end{proposition}

\begin{proof} One way to show that the map $G_{t,w}$
is one-to-one and onto is to observe that the integral equation (\ref{eq_1801}) can be reversed. Indeed, given $y \in \R^n$, the equation
\begin{equation}\label{eq_118}
\theta_s = y + w_s - w_t - \int_s^t a (r, \theta_r ) \, dr ,
\qquad s \in [0,t]
\end{equation}
also has a unique solution $(\theta_s)_{0 \leq s \leq t}$, for the same reasons that (\ref{eq_1801})
has a unique solution. Equation (\ref{eq_118}) is equivalent to the requirement that $\theta_t = y$ and that for $s \in [0,t]$,
$$ \theta_s = \theta_0 - w_0 + w_s + \int_0^s a(r, \theta_r) dr. $$
 It follows that $x:=\theta_0 - w_0$ is the
unique element of $\R^n$ such that $G_{t,w} (x) = y$. We have thus shown that the map $G_{t,w}: \RR^n \rightarrow \RR^n$ is invertible. 
Moreover, we know that $G_{t,w}$ is smooth by Lemma~\ref{lem_1800}, and the same argument applies to the
reversed equation~(\ref{eq_118}). Therefore the reciprocal of $G_{t,w}$
is also smooth. This shows that $G_{t,w}$ is
a diffeomorphism.

\medskip
For the expansion property, let $\theta_t^{x} = G_{t,w} (x)$ and
note that (\ref{eq_1329}) and (\ref{eq_1801}) imply that given $x_1,x_2\in \R^n$ we have
\[
\theta_t^{x_1} - \theta_t^{x_2} = x_1 - x_2 + \int_0^t \left[ \nabla \Lambda_s( \theta_s^{x_1}) - \nabla \Lambda_s( \theta_s^{x_2} ) \right] ds.
\]
Hence,
\[
\frac{d}{dt} \left\vert \theta_t^{x_1} - \theta_t^{x_2} \right\vert^2
= 2 \langle \nabla \Lambda_t( \theta_t^{x_1}) - \nabla \Lambda_t( \theta_t^{x_2} ) , \theta_t^{x_1}  - \theta_t^{x_2}  \rangle \geq 0,
\]
where the inequality simply follows from the convexity of $\Lambda_t$.
Thus $\vert  \theta_t^{x_1}  - \theta_t^{x_2}   \vert$ is a non decreasing
function of $t$. In particular
$\vert \theta_t^{x_1}  - \theta_t^{x_2}   \vert \geq \vert x_1 -x_2\vert$
and the proof is complete. 
\end{proof}
The next observation is that the flow has a semigroup property.
To formulate it we need to introduce further notation.
\begin{definition}
For $t_1,t_2\geq0$, $w\in \mathcal C ([0,\infty), \R^n)$ and $x\in \R^n$,
we let
\[
G_{t_1,t_2,w} (x) = \theta_{t_2}
\]
where $(\theta_t)_{t \geq 0}$ is the unique solution of
\[
\theta_t = x + w_t + \int_0^t a ( t_1 + s , \theta_s ) \, ds ,\qquad \forall t > 0.
\]
\end{definition}
Thus, the only difference with (\ref{eq_1801}) is that
we replace $a(s, \theta_s )$ by $a(t_1 + s , \theta_s )$
in the integral equation.
This amounts to replacing the reference measure $\mu$ by the
measure $\mu_{t_1,0}$.
\begin{lemma}[Semigroup property] \label{lem_6734}
Fix $w \in \mathcal C ( [0,\infty) , \R^n )$. Then for any $t_1,t_2\geq0$, 
\[
G_{t_1+t_2,w} = G_{t_1, t_2 , \sigma_{t_1}(w)} \circ G_{t_1,w},
\]
where $\sigma_{t_1}$ is the shift operator on $\mathcal C ( [0,\infty) , \R^n )$, defined by
\[
(\sigma_{t_1} (w))_t = w_{t+t_1} - w_{t_1} .
\]
\end{lemma}
\begin{proof}
Fix $x \in \R^n$, let $y = G_{t_1,w} (x)$ and $z = G_{t_2,\sigma_{t_1} (w)} (y)$. Then $y = \theta_{t_1}$ where
$(\theta_t)$ is the unique solution of
\begin{equation} \label{eq_9001}
\theta_t = x + w_t + \int_0^t a(s,\theta_s) \, ds , \qquad t \geq 0  .
\end{equation}
Similarly $z = \phi_{t_2}$ where $(\phi_{t})$ is the unique solution of
\begin{equation} \label{eq_9002}
\phi_t = y + w_{t_1+t} - w_{t_1} + \int_0^t a(t_1+s,\phi_s) \,ds , \qquad t \geq 0.
\end{equation}
Define $(\psi_t)_{t \geq 0}$ by
\[
\psi_t =
\begin{cases}
\theta_t & \text{if } t \in [0,t_1] \\
\phi_{t-t_1} & \text{if } t >  t_1 .
\end{cases}
\]
From (\ref{eq_9001}) and (\ref{eq_9002}) we  see that
$(\psi_t)$ satisfies
\[
\psi_t = x + w_t + \int_0^t a ( s, \psi_s ) \, ds , \qquad \forall t \geq 0.
\]
Therefore
\[
z = \psi_{t_1+t_2} = G_{t_1+t_2 , w} (x)  ,
\]
which is the desired result.
\end{proof}

We can parameterize the tilted measure by its barycenter rather than
by the tilt itself. The next lemma is standard (see, e.g., \cite[Lemma 2.1]{EK}), and its proof is provided
for completeness.
\begin{lemma} For any $t \geq 0$, the map $\theta \mapsto \nabla \Lambda_t(\theta)$ is a diffeomorphism from $\RR^n$ onto $K$.
\label{lem_1524}
\end{lemma}
\begin{proof}
Abbreviate $\Lambda =\Lambda_t$. 
Recall from (\ref{eq_1122}) that the Hessian matrix $\nabla^2 \Lambda(\theta) \in \RR^{n \times n}$ is the covariance matrix of a probability measure whose support spans $\RR^n$,
and is consequently a symmetric, positive definite matrix. Hence the smooth convex function 
$\Lambda: \RR^n \rightarrow \RR$ is in fact strongly convex. This already implies that
$\nabla \Lambda$ is a diffeomorphism from $\R^n$ onto
its image $\nabla \Lambda (\R^n)$ which is necessarily an open set, see e.g. \cite[section 26]{roc}.
It remains to
prove that $$ \nabla \Lambda (\R^n) = K. $$
To this end, let $L_0 \subseteq \RR^n$ be the support of $\mu$, and let $L \subseteq \RR^n$ be the  convex hull of $L_0$.
Recall that $K \subseteq \RR^n$ is the interior of $L$. 
From (\ref{eq_1329}) we see that for any $\theta \in \RR^n$, the vector $\nabla \Lambda(\theta)$ is the barycenter of a probability measure supported in the compact set $L_0$,
and thus belongs to its convex hull $L$. However, $\nabla \Lambda (\R^n)$ is an open set 
and hence it is contained in the interior of $L$. We have thus shown that $$ \nabla \Lambda (\R^n) \subseteq K. $$ For the converse inclusion
we use duality. The open set $\nabla \Lambda (\R^n)$
coincides with the interior of the domain of the Legendre conjugate of
$\Lambda$, denoted by $\Lambda^*$
(see e.g. \cite[Theorem 26.5]{roc}). It thus suffices to show that $K$
is contained in the domain of $\Lambda^*$.
In other words, we need to prove  that for any $\xi \in K$,
\[
\Lambda^* (\xi) = \sup_{\theta\in \R^n} \left[ \langle \theta ,\xi \rangle
 - \Lambda ( \xi ) \right] < + \infty.
\]
Let $\xi \in K$ and suppose by contradiction that
$\Lambda^* ( \xi ) = +\infty$.
Then there exists a sequence $\theta_1,\theta_2,\ldots \in \RR^n$ such that
\begin{equation}\label{eq_1525}
\lim_{m\to \infty}  \left[ \langle \xi , \theta_m \rangle - \Lambda (\theta_m) \right] = +\infty.
\end{equation}
Necessarily $r_m: =\vert \theta_m \vert \longrightarrow  \infty$,
and by passing to a subsequence if needed,
we may assume that $\theta_m / r_m$ converges to some unit
vector $v \in \RR^n$. 

\medskip The crucial observation is that $\Lambda (\theta_m ) /r_m$
converges to the essential supremum (with respect to $\mu$) of
the map $x\mapsto\langle x,v\rangle$. This follows from the definition of the logarithmic 
Laplace transform and a simple limiting argument. 
By the definition of the support of $\mu$,
this essential supremum coincides with $ \sup_{x\in L_0} \langle x , v\rangle$. Consequently,
\begin{equation}
\lim_{m \rightarrow \infty} \Lambda (\theta_m ) /r_m = \sup_{x\in L_0} \langle x , v\rangle.
\label{eq_1356}
\end{equation}
We know that $\langle  \xi  , \theta_m \rangle /r_m \to \langle \xi, v \rangle$
while $r_m \longrightarrow +\infty$. Thus, from (\ref{eq_1525}) and (\ref{eq_1356}), 
\[
\langle \xi , v \rangle \geq \sup_{x\in L_0} \langle x , v \rangle
 = \sup_{x\in L} \langle x , v\rangle.
\]
Thus the linear map $x\mapsto \langle x , v\rangle$ attains its maximum
on $L$ at the point $\xi \in K$. This contradicts the fact that $\xi \in K$ where $K$ is the interior of the compact, convex set $L$.
\end{proof}

By specifying Lemma \ref{lem_1524} to the case $t=0$, we see  that for any $\xi\in K$, there exists a unique $\theta\in \R^n$ for which the corresponding exponential
tilt $\mu_\theta$ has its barycenter at the point~$\xi$. 

\begin{definition}
For  $w \in \mathcal C ( [0,\infty) , \R^n)$ and $t\geq0$ define 
\[
S_{t,w} = \nabla \Lambda_t \circ G_{t,w} \circ \nabla \Lambda_0^{-1} .
\]
More generally, for $t_1,t_2 \geq 0$ we set
\[
S_{t_1,t_2,w} = \nabla \Lambda_{t_1+t_2} \circ G_{t_1,t_2,w} \circ \nabla \Lambda_{t_1}^{-1} .
\] \label{def_1942}
\end{definition}

Recall that $B=(B_t)_{t \geq 0}$ is a standard
Brownian motion in $\R^n$ with $B_0 = 0$. Consider the family of random maps $(S_t)_{t \geq 0}$ given by
\[
S_t = S_{t,B}  , \qquad \forall t \geq 0.
\]
Let us also define $S_{t_1,t_2} = S_{t_1,t_2,B}$.
The properties of the stochastic process $(S_t)_{t \geq 0}$  are summarized in the next theorem.
\begin{theorem} Let $\mu$ be a compactly-supported probability measure whose support
affinely spans $\RR^n$. Write $ K \subseteq \RR^n $ the interior of the convex hull of the support of $\mu$.
Then,
\begin{enumerate}[(a)]
\item Almost surely, for all $t \geq 0$ the random map $S_t: K \rightarrow K$ is a diffeomorphism, and $S_0 = \id$.
\item (Martingale property) For any fixed $\xi \in K$, the
process $(S_t( \xi ))_{t \geq 0}$ is a martingale. Moreover, the limit 
$$ S_{\infty}(\xi) := \lim_{t\to \infty} S_t (\xi) $$
exists almost surely, and the law of $S_{\infty}(\xi)$ is the unique exponential tilt of $\mu$ having its barycenter at the point $\xi \in K$.
\item (Markov property) For any fixed $\xi\in K$, the process
$(S_t (\xi))_{t \geq 0}$ is a time-inhomogeneous Markov process.
More precisely, for any $t_1,t_2 \geq 0$ and 
a bounded, continuous function $f \colon K \to \R$, we have
\begin{equation}\label{eq_3430}
\E [ f ( S_{t_1+t_2} (\xi) ) \mid \mathcal F_{t_1} ] = P_{t_1,t_2} f ( S_{t_1} (\xi )),
\end{equation}
where $(\mathcal F_t)$ is the natural filtration of $(B_t)$ and where
the operator $P_{t_1,t_2}$ is defined by
$$
P_{t_1,t_2}  f ( \xi ) = \E f( S_{t_1,t_2} ( \xi )) .
$$
\end{enumerate}
\label{prop_1623}
\end{theorem}

\begin{remark} Under mild regularity assumptions, and assuming that $K \subseteq \RR^n$ is strictly-convex, the diffeomorphism $S_t$ extends to a homeomorphism of the closure of $K$ which almost surely satisfies $S_t|_{\partial K} = \id$ for all $t \geq 0$. We do not prove this fact in this article.
\end{remark}

From Theorem  \ref{prop_1623}$(b)$ we see  that $S_{\infty}$ provides a simultaneous coupling of any countable subcollection of the family of exponential tilts $(\mu_x)_{x \in \RR^n}$.
We thus provide a case study
in the theory of  multi-marginal transport;  see \cite{pass} for a survey of this theory.
The proof of Theorem \ref{prop_1623} requires the following:
\begin{lemma} \label{lem_9999}
Fix $x\in \R^n$, and for $t \geq 0$ set $\theta_t = G_{t,B}(x)$ and $a_t = a(t,\theta_t)$. Then $(a_t)_{t \geq 0}$ is a martingale
and its limit as $t \to \infty$ has law $\mu_x$.
\end{lemma}
\begin{proof}
Let $X$ be a random vector having law $\mu_x$ that
is independent of the Brownian motion $(B_t)$.
By Proposition~\ref{prop_1824},
it suffices to prove that the stochastic process $(b_t)$ given by
\[
b_{t} = a (t , x + t X + B_t )  , \quad t\geq 0
\]
is a martingale whose limit as $t \to \infty$ 
equals $X$ almost surely. The  process $(b_t)_{t \geq 0}$ is uniformly bounded in view of 
(\ref{eq_1123_}). Let $X_t = t X + B_t$ and write $(\cG_t)$ for 
the natural filtration of
the process $(X_t)$. 
According to (\ref{eq_1136}),
\[
b_t = \E [ X\mid \mathcal \cG_t ] , \qquad t>0.
\]
This implies that $b_t \to \E [X\mid \mathcal G_\infty]$
almost surely, where $\mathcal G_\infty$
is the $\sigma$-algebra generated by $\cup_t \mathcal G_t$
(see e.g. \cite[Chapter 14]{will}).
However, $X = \lim_t X_t/t$ is $\mathcal G_\infty$-measurable.
Therefore $\E [ X \mid \mathcal G_\infty ] = X$ and the proof is complete. 
\end{proof}
\begin{remark}
In fact, the process $(M_t)_{t \geq 0}$ given by
$M_t = \int_{\R^n}  \phi \, d\mu_{t,\theta_t}$ is a martingale
for any bounded test function $\phi$, and not just for $\phi (x) = x$.
Hence, in a sense, the measure-valued
process $(\mu_{t,\theta_t})$ is a martingale.
This measure-valued martingale is called the \emph{stochastic localization process} associated
to $\mu$, see \cite{KLbull} and references therein.
\end{remark}
\begin{proof}[Proof of Theorem \ref{prop_1623}]
Item $(a)$ follows immediately from Proposition~\ref{prop_1456}
and Lemma~\ref{lem_1524}, since the composition of three diffeomorphisms is a diffeomorphism.
In order to prove $(b)$ we fix a point $\xi \in K$ and
let $x = (\nabla \Lambda_0)^{-1} (\xi)$, so that $\mu_{x}$ is the tilt of $\mu$
having barycenter at $\xi$. Note that
\[
S_t ( \xi )  = \nabla \Lambda_t \circ G_{t,B} ( x )
= a (t , G_{t,B} ( x ) ) .
\]
Lemma~\ref{lem_9999} thus implies $(b)$. In order to prove $(c)$, observe that by Lemma~\ref{lem_6734}
and Definition \ref{def_1942}, with $x = (\nabla \Lambda_0)^{-1} (\xi)$,
\begin{align}\label{eq_3432}
S_{t_1+t_2} (\xi) & = S_{t_1+t_2,B} (\xi)
= \nabla \Lambda_{t_1 + t_2, B} \circ G_{t_1 + t_2, B}(x)
\\ & = \nabla \Lambda_{t_1 + t_2, B} \circ G_{t_1, t_2, \sigma_{t_1}(B) } ( G_{t_1,B} (x) ) \nonumber 
\\ & = S_{t_1, t_2, \sigma_{t_1}(B) }( (\nabla \Lambda_{t_1}) \circ G_{t_1,B} (x) ) = S_{t_1, t_2, \sigma_{t_1}(B) }( S_{t_1,B}(\xi) ).
\nonumber
\end{align}
Let us now prove (\ref{eq_3430}). It follows from Lemma~\ref{lem_1300} that the process $(S_t (\xi))_{t \geq 0}$ is adapted to the filtration $(\mathcal F_{t})$
of the Brownian motion $(B_t)$.
Thus $S_{t_1}(\xi)$ is $\mathcal F_{t_1}$-measurable. Moreover,
since the Brownian motion has independent and stationary increments,
the process $\sigma_{t_1}(B)$ is
a standard Brownian motion independent of $\mathcal F_{t_1}$. Therefore
for any bounded, continuous function $f: K \rightarrow \RR$
\begin{equation}
\E [f ( S_{t_1,t_2 , \sigma_{t_1} ( B ) } ( S_{t_1,B} (\xi) ) ) \mid \mathcal F_{t_1} ] = F ( S_{t_1, B} (\xi) )  ,
\label{eq_925} \end{equation}
where $F: K \rightarrow \RR$ is given by
\[
F( \xi ) = \E f ( S_{t_1,t_2,B} (\xi )) = P_{t_1,t_2} f ( \xi ), \qquad  \xi \in K.
\]
By combining (\ref{eq_3432}) with (\ref{eq_925}) we conclude  (\ref{eq_3430}).
\end{proof}

\begin{remark} The Markov process $(S_{t}(\xi))$ is  a time-inhomogeneous
diffusion, whose generator is the second order
differential operator $\mathcal L_t$ given by
\[
\mathcal L_t f (\xi)  = \frac 12 \Tr \left[
\nabla^2 \Lambda_t \left( (\nabla \Lambda_t)^{-1} (\xi) \right)  \nabla^2 f ( \xi )
\right] ,
\]
for suitable functions $f \colon K \to \R$.
The proof is omitted.
\end{remark}

%
%
%
%
%
%

\section{Wasserstein distances in the log-concave case}
\label{sec3}

As in the previous sections, let $\mu$ be a compactly-supported probability measure whose support affinely spans $\RR^n$.
In this section we add the assumption that $\mu$ is {\it log-concave}. In this case, the log-concave Lichnerowicz inequality (e.g. \cite[Section 4]{KLbull} and references therein)
implies that for any $t > 0$ and $\theta \in \RR^n$,
\begin{equation}  A(t, \theta) = \nabla^2 \Lambda_t(\theta) \leq \frac{1}{t} \cdot \id \label{eq_1817} \end{equation}
in the sense of symmetric matrices. By integration, this implies that for any $t > 0$ and $\theta_1, \theta_2 \in \RR^n$,
\begin{equation}  \langle \nabla \Lambda_t(\theta_1) - \nabla \Lambda_t(\theta_2), \theta_1 - \theta_2 \rangle \leq \frac{1}{t} \cdot |\theta_1 - \theta_2|^2.
\label{eq_1734} \end{equation}
Recall the flow map $G_{t,w}: \RR^n \rightarrow \RR^n$ from Definition \ref{def_1607}.
\begin{lemma} \label{lem_1119}
If $\mu$ is log-concave, then for any $w\in\mathcal C ([0,\infty) , \R^n)$
and $x,y \in \R^n$, the quantity
\begin{equation} 
\frac{ \vert G_{t,w} (x) - G_{t,w} (y) \vert }{t}
\label{eq_1144} \end{equation}
is a non-increasing function of $t \in (0, \infty)$.
\end{lemma}

\begin{proof}
The proof is  similar to the second part of the proof of
Proposition \ref{prop_1456}.
Let $\theta_t^{x} = G_{t,w} (x)$ and $\theta_t^{y} = G_{t,w} (y)$.
By (\ref{eq_1329}) and (\ref{eq_1801}), 
\[
\theta_t^x - \theta_t^y = x-y
+ \int_0^t \left[ \nabla \Lambda_s( \theta_s^x )
- \nabla \Lambda_s ( \theta_s^y ) \right] \, ds .
\]
Differentiating with respect to $t$ and using (\ref{eq_1734}),
\[
 \frac{d}{dt} \vert \theta_t^x - \theta_t^y \vert^2
= 2 \langle \nabla \Lambda_t(\theta_t^x)
- \nabla \Lambda_t( \theta_{t}^y ) , \theta_t^x - \theta_t^y \rangle \\
 \leq \frac{2}{t} \cdot \vert \theta_t^x - \theta_t^y \vert^2 .
 \]
Hence
\[
\frac d{dt} \frac{ \vert \theta_t^x - \theta_t^y \vert^2 }{ t^2 } \leq 0.
\]
This implies that the function in (\ref{eq_1144}) is non-increasing in $t$. 
\end{proof}

For two Borel probability measures $\nu_1, \nu_2$ in $\RR^n$ and for $1 \leq p < \infty$ we write $W_p(\nu_1, \nu_2)$ for the $L^p$-Wasserstein distance between $\nu_1$ and $\nu_2$. That is,
$$ W_p(\nu_1, \nu_2) = \inf_{X_1,X_2} \left( \EE |X_1 - X_2|^p \right)^{1/p} $$
where the infimum runs over all random vectors $X_1,X_2$ defined on the same probability space with $X_i$ having law $\nu_i$ for $i=1,2$. In other words, $X_1$ and $X_2$ provide a  coupling of $\nu_1$ and $\nu_2$. As before, we let $B = (B_t)_{t\geq 0}$
be a standard Brownian motion in $\RR^n$, with $B_0 = 0$.

\begin{proposition} \label{prop_1705}
Assume that $\mu$ is log-concave. For $x\in \R^n$
and $t >0$ set $\theta_t^x = G_{t,B} ( x)$. Then for any
$x,y \in \R^n$, $1 \leq p < \infty$ and $t > 0$, 
 \begin{equation} \label{eq_1424} W_p \left(\mu_{x}, \mu_{y} \right) \leq \frac{1}{t} \cdot \left( \EE \left| \theta_t^x - \theta_t^y \right|^p \right)^{1/p}. \end{equation}
\end{proposition}

\begin{proof} By Corollary \ref{cor_1000} we know that
\[
\lim_{t\to \infty} \frac{\theta_t^{x}}{t}
\]
exists almost surely, and has law $\mu_x$. Similarly $\lim_{t} \theta_t^{y} / t$ exists and has law $\mu_y$. Thus, by the definition of the Wasserstein distance
\[
W_p \left(\mu_{x}, \mu_{y} \right)^p \leq
\EE  \left\vert \lim_{t\to \infty} \frac{\theta_t^{x}}{t} - \lim_{t \to \infty} \frac{ \theta_t^{y} }{t} \right\vert^p 
=
\EE \lim_{t\to \infty} \left\vert \frac{\theta_t^{x} -\theta_t^{y} }{t} \right\vert^p .
\]
On the other hand, the quantity
$ \left\vert \frac{\theta_t^{x} - \theta_t^{y}}{t} \right \vert$
is almost surely a non-increasing function of $t \in (0, \infty)$, according 
to Lemma~\ref{lem_1119}.
In particular, the value of this quantity at any fixed time $t$
is at least as large as the limit value, and (\ref{eq_1424}) follows.
\end{proof}
Next we formulate an infinitesimal version of Proposition \ref{prop_1705} in the case $p=2$.
Recall from Lemma~\ref{lem_1800}
that for any $t>0$ and any continuous path $w$,
the map $G_{t,w}: \RR^n \rightarrow \RR^n$ is smooth and  $G_{t,w}'(x) \in \RR^{n \times n}$
denotes its derivative at the point $x \in \RR^n$.
%
\begin{corollary} \label{cor_1521}
Assume that $\mu$ is log-concave. For $t \geq 0$ set $$ M_t = G_{t, B}'(0) \in \RR^{n \times n}. $$ 
Then for any $v \in \RR^n$ and $t > 0$,
\begin{equation}
\limsup_{\epsilon \to 0^+} \frac{ W_2 ( \mu , \mu_{\epsilon v} ) }\epsilon
\leq \frac{  ( \E \vert M_t v \vert^2 )^{1/2} } t.
\label{eq_1521}
\end{equation}
\end{corollary}

\begin{proof}
Recall that $\mu = \mu_0$ and that we use the notation $\theta_t^x = G_{t,B} (x)$.
By Lemma~\ref{lem_1400}, almost surely, for all  $\epsilon > 0$,
\[
\frac{ \vert \theta_t^0 - \theta_t^{\epsilon v} \vert }{ \epsilon } \leq \e^{C_\mu t} |v|,
\]
for some constant $C_\mu > 0$.
Thus, by the dominated convergence theorem, 
\[
\lim_{\epsilon \to 0^+}
\frac{ \E \vert \theta_t^0 - \theta_t^{\epsilon v} \vert^2 }{\epsilon ^2 } =
 \E \left\vert \lim_{\epsilon \to 0^+} \frac{\theta_t^0 - \theta_t^{\epsilon v}}{\epsilon}  \right\vert^2  = 
 \E \vert \partial_v G_{t,B}(0) \vert^2 =  \E \vert M_t v\vert^2 .
\]
By substituting this into Proposition \ref{prop_1705} we obtain (\ref{eq_1521}). 
\end{proof}

The infinitesimal Wasserstein distance
is intimately related to the $H^{-1}$-norm,
see e.g. Villani \cite[Section 7.6]{Vil} or the Appendix of \cite{K_ptrf}.
Specifically, we shall need the following lemma:

\begin{lemma}\label{lem_villani} Let $\mu$ be a centered, compactly-supported
probability measure on $\RR^n$. Then for any vector $v\in \R^n$,
\[
\Vert \langle x ,v\rangle \Vert_{H^{-1} (\mu) }
\leq \limsup_{\epsilon \to 0^+} \frac{ W_2 ( \mu , \mu_{\epsilon v} ) }{\epsilon} .
\]
\end{lemma}

\begin{proof} 
As usual, we write $o(\eps)$ for an expression $X$ such that $X / \eps$ tends to zero as $\eps \to 0$,
while $o(1)$ stands for an expression $X$ that itself tends to zero as $\eps \to 0$. We may assume that
\begin{equation} \label{eq_ffrriznvoen}
W^2_2 (\mu ,\mu_{\epsilon v} )= o(\epsilon) ,
\end{equation}
since otherwise the conclusion of the lemma is vacuous. 
The measure $\mu$ is centered, and from (\ref{eq_1329}) we see that $\nabla \Lambda_0(0) = 0$ and $\Lambda_0(0) = 0$.
Consequently, as $\eps \rightarrow 0$, $$ \Lambda_0(\eps e_i) = o(\eps). $$
Fix a smooth, compactly-supported function $\vphi: \RR^n \rightarrow \RR$. Since $\vphi$ is compactly-supported, for $\eps > 0$,
\begin{equation} \label{eq_1010}
\begin{split}
\int_{\RR^n} \langle x,v\rangle \vphi(x) \, d \mu(x) &
= \int_{\RR^n} \frac{e^{\eps \langle x, v\rangle} - 1}{\eps} \vphi(x)
\, d \mu(x) + o(1)
\\ & = \int_{\RR^n} \frac{e^{\eps \langle x,v\rangle - \Lambda_0(\eps v)} - 1}{\eps}  \vphi(x) \, d \mu( x) + o(1)  \\
& = \frac{1}{\eps} \left[ \int_{\RR^n} \vphi \, d \mu_{\epsilon v} - \int_{\RR^n} \vphi \, d \mu \right] + o(1) .
\end{split}
\end{equation}
Since $\vphi$ is smooth and compactly-supported, by Taylor's theorem there exists $R_0 > 0$ such that for all $x,y \in \RR^n$,
\begin{equation}  \vphi(y) - \vphi(x) \leq |\nabla \vphi(x)| \cdot |y -x| + R_0 |x-y|^2. \label{eq_1715} \end{equation}
Let us momentarily fix $\eps > 0$ and let $X,Y$ be an arbitrary coupling of $\mu$ and $\mu_{\eps v}$, i.e., $X$ has law $\mu$
and $Y$ has law $\mu_{\eps v}$. 
By (\ref{eq_1715}) and the Cauchy-Schwartz inequality,
\[
\begin{split}
\left\vert \int_{\R^n} \phi \, d \mu
- \int_{\R^n} \phi \, d\mu_{\epsilon v} \right\vert &
= \left\vert \EE \vphi(X) - \vphi(Y)  \right\vert \\
& \leq \EE \left[ |\nabla \vphi(X)| \cdot |Y -X| + R_0 |X-Y|^2 \right] \\
& \leq \Vert \nabla \phi \Vert_{L^2(\mu)}
\cdot \sqrt{ \EE |X- Y|^2 } + R_0 \cdot \EE |X- Y|^2.
\end{split}
\]
By considering the infimum over all couplings $X, Y$, we conclude  that for any $\eps > 0$,
$$ \left\vert \int_{\R^n} \phi \, d \mu
- \int_{\R^n} \phi \, d\mu_{\epsilon v} \right\vert 
\leq  \Vert \nabla \phi \Vert_{L^2(\mu)}
\cdot W_2 (\mu , \mu_{\epsilon v} ) + R_0 \cdot W_2 (\mu , \mu_{\epsilon v} )^2. $$
By substituting this back in~(\ref{eq_1010}) we obtain 
\[
\int_{\R^n} \langle x,v\rangle \phi (x) \, d \mu(x)
\leq \epsilon^{-1} \Vert \nabla \phi \Vert_{L^2 (\mu)}
\cdot W_2 (\mu , \mu_{\epsilon v} )+ \epsilon^{-1} R_0
\cdot W_2^2 (\mu, \mu_{\epsilon v} ) + o (1) .
\]
By letting $\epsilon$ tend to $0$ and using (\ref{eq_ffrriznvoen})  we conclude that
for any compactly-supported, smooth function $\phi: \RR^n \rightarrow \RR$,
\[
\int_{\R^n} \langle x,u\rangle \phi (x) \, d \mu(x)
\leq \Vert \nabla \phi \Vert_{ L^2 (\mu) } \cdot
\limsup_{\epsilon \to 0^+} \frac{ W_2 ( \mu , \mu_{\epsilon v} ) }{\epsilon}.
\]
This completes the proof, thanks to the definition (\ref{eq_1528}) of 
the $H^{-1}(\mu)$-norm.
\end{proof}
By applying Lemma~\ref{lem_villani}
for the coordinate vectors $e_1,\dotsc,e_n \in \RR^n$
and combining its conclusion with Corollary~\ref{cor_1521} we
arrive at the following:
\begin{corollary} \label{cor_4701}
Assume that $\mu$ is centered, compactly-supported and log-concave. Then for any $t > 0$,
\[
\sum_{i=1}^n \Vert x_i \Vert_{H^{-1}(\mu)}^2
\leq \frac 1{t^2} \cdot \E \vert M_t \vert^2, 
\]
where $M_t = G_{t,B}'(0) \in \RR^{n \times n}$ and
\[
\vert M_t \vert = \left(\sum_{i=1}^n \vert M_t e_i\vert^2 \right)^{1/2} = \left( \tr[ M_t^* M_t ] \right)^{1/2}
\]
is the Hilbert-Schmidt norm of $M_t$.
\end{corollary}
Most of the remainder of this paper is devoted to estimating $\E \vert M_t \vert^2$ from above. 
Fix a  path $w\in \mathcal C ( [0,\infty) , \R^n )$
and a point $x\in \R^n$, and let $M_t = G_{t,w}'(x)$. Recall from Lemma \ref{lem_1800}
that $(M_t)_{t \geq 0}$ is a $\mathcal C^1$-smooth
function, and that it is the unique solution of the equation
\begin{equation} \label{eq_1216}
\begin{cases}
M_0 = \id \\
\frac d{dt} M_t = A_t M_t
\end{cases}
\end{equation}
where $A_t = A(t, G_{t,w} (x) )$. This equation is
sometimes referred to as the \emph{product integral equation}. In dimension $1$, the solution of (\ref{eq_1216}) is
simply $M_t = \exp ( \int_0^t A_s \,ds )$. This identity does not necessarily hold
in higher dimensions, due to the lack of commutativity,
but nevertheless we  have the following inequality:
\begin{proposition}\label{prop_1813}
Let $(A_t)_{t \geq 0}$ be a continuous path
of symmetric, positive-definite $n \times n$ matrices, and let $(M_t)_{t \geq 0}$ be the
 solution of (\ref{eq_1216}). Denote
the eigenvalues $A_t$, repeated according to their multiplicity, by $\lambda_{1}(t) \geq \ldots \geq \lambda_{n}(t) > 0$. Then for any $t > 0$,
\begin{equation} \vert M_t \vert^2 \leq \sum_{i=1}^n \exp \left( 2 \int_0^t \lambda_{i}(s) ds \right) .
\label{eq_1357} \end{equation}
\end{proposition}

The proof of Proposition \ref{prop_1813} requires the following lemma.

\begin{lemma} Let $\mu_1(t), \ldots, \mu_n(t)$ and $\lambda_1(t) \geq \ldots \geq \lambda_n(t)$ be non-negative, continuous functions of $t \in [0, \infty)$.
Assume that  for $t \geq 0$ and $k=1,\ldots,n$,
\begin{equation}\label{eq_stepmlljhbv_}
\sum_{i=1}^k \mu_i(t)
\leq \sum_{i=1}^k \left[ 1 + 2 \int_0^t \mu_i (s) \lambda_i (s) \, ds \right].
\end{equation}
Then for $t \geq 0$ and $k=1,\ldots,n$, 
\begin{equation}\label{eq_nvfinvkfvvvff}
\sum_{i=1}^k \mu_i (t) \leq \sum_{i=1}^k \exp \left( 2 \int_0^t \lambda_i (s) \, ds \right).
\end{equation}
\label{lem_2230}
\end{lemma}

\begin{proof} Denote
\begin{equation}  \nu_i(t) = 1 + 2 \int_0^t \mu_i (s) \lambda_i (s) \, ds. \label{eq_1917} \end{equation}
According to (\ref{eq_stepmlljhbv_}), for all $k$ and $t$,
\begin{equation}\label{eq_stepmlljhbv}
\sum_{i=1}^k \mu_i(t)
\leq \sum_{i=1}^k \nu_i(t).
\end{equation}
We will prove (\ref{eq_nvfinvkfvvvff}) by induction on $k$. Consider first the case $k=1$.
Note that $\nu_1$ is $\mathcal C^1$-smooth in $t$, and
that by (\ref{eq_1917}) and the case $k=1$  of (\ref{eq_stepmlljhbv}) we have
\[
\frac d{dt} \nu_1 (t) = 2 \mu_1 (t) \lambda_1 (t)
\leq 2 \nu_1(t) \lambda_1 (t) .
\]
By integrating this differential inequality and
using ~\eqref{eq_stepmlljhbv} we obtain 
\[
\mu_1 (t) \leq \nu_1 (t) \leq
\exp \left( 2 \int_0^t \lambda_1 (s) \, ds \right).
\]
This is precisely the case $k=1$ of the desired inequality (\ref{eq_nvfinvkfvvvff}).
Next, let $k \geq 2$ and assume that (\ref{eq_nvfinvkfvvvff})
holds true all the way up
to $k-1$. Observe that since
\[
\lambda_1 (t) -\lambda_k (t) \geq \lambda_2(t) - \lambda_k(t) \geq \dotsb \geq \lambda_{k-1} (t) - \lambda_k (t) \geq 0 ,
\]
the induction hypothesis implies that
\[
\sum_{i=1}^{k-1} \mu_i (t) (\lambda_i (t) - \lambda_k (t) )
\leq \sum_{i=1}^{k-1} \exp \left( 2 \int_0^t \lambda_i(s) \,ds  \right)
(\lambda_i (t) - \lambda_k (t) ) .
\]
By combining this with (\ref{eq_1917}) and \eqref{eq_stepmlljhbv} we obtain
\[
\begin{split}
\frac{d}{dt} \sum_{i=1}^k \nu_i (t)
& = 2 \sum_{i=1}^k \mu_i (t) \lambda_i (t) \\
& = 2 \sum_{i=1}^{k-1} \mu_i (t) (\lambda_i (t) - \lambda_k (t) )
 + 2 \left( \sum_{i=1}^k \mu_i (t) \right) \lambda_k (t)  \\
& \leq 2 \sum_{i=1}^{k-1} \exp \left( 2 \int_0^t \lambda_i (s) \, ds \right)
 ( \lambda_i (t) - \lambda_k (t) ) + 2 \left( \sum_{i=1}^k \nu_i (t) \right) \lambda_k (t) .
\end{split}
\]
Elementary manipulations show that the last inequality can be reformulated
as
\[
\frac d{dt} \left( \exp \left( -2 \int_0^t \lambda_k (s) \, ds \right) \sum_{i=1}^k \nu_i (t) \right) \leq
\frac{d}{dt} \left( \sum_{i=1}^{k-1} \exp \left( 2 \int_0^t (\lambda_i (s) -\lambda_k (s)) \, ds \right) \right) .
\]
Integrating, and recalling that $\nu_i (0) = 1$ for  $1 \leq i\leq n$
we obtain 
\[
\exp \left( -2 \int_0^t \lambda_k (s) \, ds \right) \sum_{i=1}^k \nu_i (t)  \leq 1 +
\sum_{i=1}^{k-1} \exp \left( 2 \int_0^t (\lambda_i (s) -\lambda_k (s)) \, ds \right) .
\]
Recalling (\ref{eq_stepmlljhbv}) we finally deduce that 
\[
\sum_{i=1}^k \mu_i (t) \leq \sum_{i=1}^k \nu_i (t) \leq
\sum_{i=1}^k \exp \left( 2 \int_0^t \lambda_i (t) \, ds \right).
\]
This completes the proof.
\end{proof}

Let $A, B \in \RR^{n \times n}$ be symmetric, positive semi-definite matrices,
and write $a_1 \geq \ldots \geq a_n$ for the eigenvalues of $A$ while $b_1 \geq \ldots \geq b_n$ are the eigenvalues of $B$.
Then,
\begin{equation}  \Tr[A B] \leq \sum_{i=1}^n a_i b_i.
\label{eq_2124} \end{equation}
This inequality is proven e.g. in \cite[Theorem 8.7.6]{HJ}, where it is referred to as the von Neumann trace inequality.

\begin{proof}[Proof of Proposition \ref{prop_1813}]
Since $A_t$ is a symmetric matrix, it follows from (\ref{eq_1216}) that
$$ \frac{d}{dt} M_t^* M_t = 2 M_t^* A_t M_t. $$
Denote the eigenvalues of $M_t^* M_t$ by $\mu_1(t) \geq \ldots \geq \mu_n(t) \geq 0$,
and note that these are also the eigenvalues of $M_t M_t^*$.
Let $1 \leq k \leq n$ and let $P \in \RR^{n \times n}$ be an orthogonal projection matrix  of rank $k$.
Write $p_1(t) \geq \ldots \geq p_n(t) \geq 0$ for the eigenvalues of $M_t P M_t^*$.
Then,
\begin{align}  \nonumber \frac{d}{dt} \tr[ M_t^* M_t P] & = 2 \tr[ M_t^* A_t M_t P] = 2 \tr[ A_t (M_t P M_t^*) ] \\ & \leq 2 \sum_{i=1}^n \lambda_i(t) p_i(t)
= 2 \sum_{i=1}^k \lambda_i(t) p_i(t), \label{eq_2220} \end{align}
according to (\ref{eq_2124}), where we note that $p_i(t) = 0$ for $i > k$ since the matrix $M_t P M_t^*$ has rank at most $k$.
By the min-max characterization of the eigenvalues of a symmetric matrix,
\begin{equation}\label{eq_2221}
\begin{split}  p_i(t) & = \max_{E \in G_{n,i}} \min_{0 \neq v \in E} \frac{\langle M_t P M_t^* v, v \rangle}{|v|^2}
= \max_{E \in G_{n,i}} \min_{0 \neq v \in E} \frac{|P M_t^* v|^2}{|v|^2}
\\ & \leq \max_{E \in G_{n,i}} \min_{0 \neq v \in E} \frac{|M_t^* v|^2}{|v|^2} = \max_{E \in G_{n,i}} \min_{0 \neq v \in E} \frac{\langle M_t M_t^* v, v \rangle}{|v|^2} = \mu_i(t),
\end{split}
\end{equation}
where $G_{n,i}$ is the collection of all $i$-dimensional subspaces of $\RR^n$.
Recall that $M_0 = \id$. Hence, by integrating  (\ref{eq_2220}) and using (\ref{eq_2221}),
\begin{equation}
\tr[ M_t^* M_t P ] \leq \tr[P] + 2 \sum_{i=1}^k \int_0^t \lambda_i(s) \mu_i(s) ds
= \sum_{i=1}^k \left[ 1 + 2 \int_0^t \lambda_i(s) \mu_i(s) ds \right].
\label{eq_2223}
\end{equation}
Inequality (\ref{eq_2223}) is valid in particular for the orthogonal projection $P$
onto the span of the $k$ eigenvectors of $M_t^* M_t$ that correspond to the eigenvalues $\mu_1(t),\ldots,\mu_k(t)$. It thus follows from (\ref{eq_2223}) that
for $k=1,\ldots,n$,
\begin{equation}
\sum_{i=1}^k \mu_i(t) \leq \sum_{i=1}^k \left[ 1 + 2 \int_0^t \lambda_i(s) \mu_i(s) ds \right].
\label{eq_2224}
\end{equation}
Inequality (\ref{eq_2224}) is precisely the assumption (\ref{eq_stepmlljhbv_})
of Lemma \ref{lem_2230}. Since $A_t$ and $M_t^* M_t$ vary continuously with $t$,
the eigenvalues $\lambda_1(t) \geq \ldots \geq \lambda_n(t) \geq 0$ and 
$\mu_1(t) \geq \ldots \geq \mu_n(t)$ vary continuously with $t$ as well.
We may therefore apply Lemma \ref{lem_2230} with $k =n$
and conclude that
$$ \Tr[M_t^* M_t] = \sum_{i=1}^n \mu_i (t) \leq \sum_{i=1}^n \exp \left( 2 \int_0^t \lambda_i (s) \, ds \right), $$
which is the desired inequality (\ref{eq_1357}).
\end{proof}

\begin{remark} We may deduce from (\ref{eq_1357}) through Jensen's inequality the arguably simpler bound
\[
|M_t|^2 \leq \frac{1}{t} \int_0^t \Tr \left[ e^{2 t A_s} \right] ds.
\]
However, it is the more sophisticated inequality (\ref{eq_1357}) that is needed for the proof below.
\end{remark}
\begin{remark} \label{rem_1156} Let $\mu$ be a log-concave probability measure in $\RR^n$, and
let $\gamma$ be the standard Gaussian probability
	measure in $\RR^n$. By using the couplings discussed above and the bound $A_t \leq \id / t$, one may prove a peculiar bound about the
	exponential tilts of $\mu$ and those of $\mu * \gamma$. Namely, for any $\theta_1 , \theta_2 \in \R^n$ and $p \geq 1$ we have the bound
$$
	W_p \left( \mu_{\theta_1}, \mu_{\theta_2} \right) \leq W_p \left( (\mu * \gamma)_{\theta_1}, (\mu* \gamma)_{\theta_2} \right) .
		$$
We omit the details of the proof.
\end{remark}
By combining  Corollary \ref{cor_4701} and
Proposition \ref{prop_1813} we  obtain the following:
\begin{corollary}\label{cor_311} Let $\mu$ be a compactly-supported, centered, log-concave probability measure in $\RR^n$.
Then for any fixed $t > 0$,
\begin{equation}  \sum_{i=1}^n \| x_i \|_{H^{-1}(\mu)}^2 \leq \frac{1}{t^2} \cdot \EE \left[  \sum_{i=1}^n  \exp \left( 2 \int_0^t \lambda_{i} (s) ds \right) \right], 
\label{eq_1933} \end{equation}
where $\lambda_{1} (t) \geq \dotsb \geq \lambda_n (t) >0$
are the eigenvalues of the matrix
\[
A_t = \nabla^2 \Lambda_t ( G_{t,B} (0) ) ,
\]
and $B = (B_t)_{t \geq 0}$ is a standard Brownian motion in $\RR^n$ with $B_0 = 0$.
\end{corollary}


\section{The covariance process of stochastic localization}
\label{sec4}

Let $\mu$ be an  isotropic, compactly-supported, log-concave probability measure in $\RR^n$,
and let $(B_t)_{t \geq 0}$ be a standard Brownian motion in $\R^n$ with $B_0 = 0$.
We let $(\theta_t)_{t \geq 0}$ be the stochastic process given by
\[
\theta_t = G_{t,B} (0) , \qquad  t \geq 0.
\]
The measure-valued process $(\mu_t)_{t \geq 0}$ defined via
\[
\mu_t = \mu_{t,\theta_t} , \qquad t \geq 0
\]
is called \emph{the stochastic localization process} of $\mu$.
Since $\mu$ is
log-concave, almost surely for any $t > 0$ the probability measure $\mu_t$ is  $t$-uniformly log-concave (see e.g. \cite{K_root} for the simple explanation). As before, we let $a_t = a(t, \theta_t)$ be the barycenter process:
\[
a_t = \int_{\R^n} x \, d \mu_t (x) = \nabla \Lambda_t ( \theta_t ) ,
\qquad  t\geq 0,
\]
while $A_t = A(t, \theta_t)$ is the covariance process:
\[
A_t = \cov ( \mu_t ) = \nabla^2 \Lambda_t ( \theta_t ) , \qquad t\geq  0.
\]
%
Recall that we denote by $\lambda_1 (t) \geq \dotsb \geq \lambda_n (t)>0$ the eigenvalues of $A_t$. In Guan~\cite{guan} it is proved that
\begin{equation}\label{eq_guan}
\E \Tr A_t^2 = \E \left[ \sum_{i=1}^n \lambda_i (t)^2 \right] \leq C n , \qquad \forall t >0 .
\end{equation}
In this section we prove a result of the same flavor, which reads as follows:
\begin{proposition} \label{prop_vnis}
Let $\tau$ be a stopping time, with respect to the natural filtration of the Brownian motion $(B_t)_{t \geq 0}$.
Then for any fixed $t > 0$,
\[
\sum_{i=1}^n \prob ( \lambda_i (t\wedge \tau ) \geq 3 )
\leq C n \cdot \exp ( - t^{-\alpha} ) ,
\]
where $C, \alpha > 0$ are universal constants.
Our proof yields $\alpha = 1/8$.
\end{proposition}

Consider the stopping time 
\begin{equation}\label{eq_deftaustar}
\tau_* = \inf \{ t > 0 \, ; \,  \Vert A_t \Vert_{op} \geq 2 \} ,
\end{equation}
where $\| \cdot \|_{op}$ is the operator norm, i.e., $\Vert A_t \Vert_{op} = \lambda_1 (t)$.
It is known (see for instance \cite[Section 7]{KLbull} or references therein) that
\begin{equation}
\label{eq_1737}
\prob ( \tau_* \leq t ) \leq \exp (- c\cdot t^{-1} ) ,
\qquad \forall t \leq c(\log n)^{-2}.
\end{equation}
We do not need the full strength of the estimate (\ref{eq_1737}) in our argument below. Rather,
we will use a much simpler qualitative fact, that 
\begin{equation}\label{eq_1738}
\prob ( \tau_* \leq t ) = o (t^k) , \qquad \forall k \geq 1. 
\end{equation}
Nevertheless, note that if~\eqref{eq_1737} were  true
for any time $t$, and not just in the range $[0,c \cdot (\log n)^{-2}]$,
then Proposition~\ref{prop_vnis} would follow
from the obvious inequalities
\[
\prob ( \lambda_i (t \wedge \tau ) \geq 3 ) \leq
\prob ( \lambda_i (t \wedge \tau ) \geq 2 )
\leq \prob ( \Vert A_{t\wedge \tau} \Vert_{op} \geq 2 )
\leq \prob ( \tau_* \leq t ) .
\]
However, such an optimistic estimate
for the operator norm of $A_t$ cannot be true in general,
see~\cite[section 8.1]{KLbull}.  Thus, short-time bounds such as (\ref{eq_1737}) are inadequate 
for proving Proposition~\ref{prop_vnis}, 
and we need to use {\it growth regularity estimates}, which are estimates showing that the matrix $A_t$ cannot grow too wildly on small intervals. 
Such estimates were established in Chen \cite{chen} and later in Guan \cite{guan}, and 
the proof of Proposition \ref{prop_vnis} 
relies heavily on \cite{guan}. In fact, in the case where $\tau \equiv +\infty$, the conclusion of Proposition \ref{prop_vnis} 
essentially follows from Guan's argument in \cite{guan}.

\medskip The proof of Proposition \ref{prop_vnis} occupies the remainder of this section. 
We begin with Guan's bound on $3$-tensors from \cite{guan}, whose proof is provided for completeness:
\begin{lemma}\label{lem_guan} Let $t > 0$ and suppose that $X$ is a centered, $t$-uniformly log-concave random vector in $\R^n$. 
Let $\lambda_1,\ldots,\lambda_n \in \RR$ be the eigenvalues of $\cov(X)$
and let $v_1,\ldots,v_n \in \RR^n$ be a corresponding orthonormal basis of eigenvectors. 
Abbreviate $X_i = \langle X, v_i \rangle$. Then for  $1 \leq k \leq n$ and $u>0$,
$$ \sum_{i,j=1}^{n}  ( \E X_i X_j X_k )^2 \mathbbm 1_{\{ \max(\lambda_i,\lambda_j) \leq u \}}  \leq 4 t^{-1/2} u^{3/2} \lambda_k. $$
\end{lemma}
\begin{proof} Write $E \subseteq \RR^n$ for the subspace spanned by the vectors
$v_i$ for which $\lambda_i \leq u$. Let $Proj_E$ be the orthogonal projection operator onto $E$ in $\RR^n$.
Note that
\begin{equation} 
\sum_{i,j=1}^{n}  ( \E X_i X_j X_k )^2  \mathbbm 1_{\{ \max(\lambda_i,\lambda_j) \leq u \}}
= \Tr ( H^2 )
\label{eq_1055} \end{equation}
where $H = \E \left[ X_k Y \otimes Y \right] \in \RR^{n \times n}$ and $Y = Proj_E X$.
It follows from the Pr\'ekopa-Leindler inequality that
$Y$ is also $t$-uniformly log-concave.
By the definition of the subspace $E$, we have $\Vert \cov(Y) \Vert_{op} \leq u$.
Thanks to the improved Lichnerowicz inequality from~\cite{K_root}, the Poincar\'e constant of $Y$ (denoted by $C_P (Y)$) satisfies
\[
C_P ( Y ) \leq  \sqrt{ \frac u t }.
\]
Consequently, 
\begin{equation}\label{eq_nivmrvfilf}
\begin{split}
\var ( \langle H Y, Y\rangle )
& \leq C_P(Y) \cdot  \E \vert 2 H Y \vert^2  \\
& \leq 4 t^{-1/2} u^{1/2} \cdot \Tr ( H^2 \cov (Y) ) \\
& \leq 4 t^{-1/2} u^{3/2} \cdot \Tr H^2 .
\end{split}
\end{equation}
On the other hand, since $X_k$ has mean $0$, the Cauchy-Schwarz
inequality shows that
\begin{equation} \label{eq_vommlg}
\begin{split}
\Tr ( H^2 ) & = \E X_k \langle H Y , Y \rangle  \\
& \leq ( \E X_k^2 )^{1/2} \cdot ( \var \langle H Y , Y \rangle )^{1/2} \\
& = \lambda_k^{1/2} \cdot ( \var \langle H Y , Y \rangle )^{1/2} .
\end{split}
\end{equation}
The conclusion of the lemma follows from (\ref{eq_1055}), (\ref{eq_nivmrvfilf}) and (\ref{eq_vommlg}).
\end{proof}

Recall that $\lambda_1(t) \geq \ldots \geq \lambda_n (t) > 0$ are the eigenvalues of  $A_t$.
Let 
\[ 
u_1(t),\ldots,u_n(t) \in \R^n 
\]
be a corresponding orthonormal basis of eigenvectors.
For $i,j=1,\ldots,n$ denote
\begin{equation} \label{eq_4700}
\xi_{ij}(t) = \int_{\R^n} \langle x - a_t, u_i(t)\rangle
\langle x - a_t , u_j(t)\rangle \,  (x - a_t)  \, d\mu_t(x) \,  \in \R^n,
\end{equation}
and
$$ \xi_{ijk}(t) = \int_{\R^n} \langle x - a_t, u_i(t) \rangle
\langle x - a_t, u_j(t) \rangle
\langle x - a_t , u_k(t) \rangle \, d\mu_t(x) \, \in \R.
$$
For a smooth function $f: \R \rightarrow \R$ and a symmetric matrix $A \in \RR^{n \times n}$
whose spectral decomposition is $A = \sum_{i=1}^n \lambda_i \, u_i \otimes u_i$
we set $f(A)= \sum_{i=1}^n f(\lambda_i) \, u_i \otimes u_i$. In particular
$$ \Tr f(A) = \sum_{i=1}^n f(\lambda_i) . $$

\begin{lemma}
For any $\mathcal C^2$-smooth function $f: [0, \infty) \rightarrow \R$
and any stopping time $\tau$ we have
\begin{equation}
\begin{split}
\frac{d}{dt} \E\Tr f(A_{t\wedge \tau})
& = \frac{1}{2} \sum_{i,j=1}^n \E  \left[  |\xi_{ij}(t)|^2
\frac{f'(\lambda_i(t)) - f'(\lambda_j(t))}{\lambda_i - \lambda_j}
\cdot \mathbbm 1_{\{t < \tau\}} \right] \\
& - \E \left[ \sum_{i=1}^n \lambda_i^2 f'(\lambda_i(t)) \cdot \mathbbm 1_{\{t < \tau\}} \right],
\end{split}
\label{eq_1016} \end{equation}
where we interpret the quotient by continuity as $f''(\lambda_i(t))$
when $\lambda_i (t) = \lambda_j (t)$.
\label{lem_1413}
\end{lemma}

\begin{proof}
It is known that the matrix-valued process $(A_t)_{t \geq 0}$ satisfies the equation 
\begin{equation} 
d A_t =  \sum_{i=1}^n H_{i,t} d B_{i,t} - A_t^2 \, dt
\label{eq_2000} \end{equation}
where $B_{1,t},\dotsc,B_{n,t}$ are the coordinates of the
Brownian motion $(B_t)$, and where $(H_{i,t})_{t \geq 0}$ is the matrix-valued process given by
\begin{equation}\label{eq_4701}
H_{i,t} = \int_{\R^n} (x_i-a_{i,t} ) \, (x-a_t)^{\otimes 2} \, d \mu_t(x),
\end{equation}
with $a_t = (a_{t,1},\ldots, a_{t,n}) \in \RR^n$. See e.g. \cite[section 7]{KLbull} for a derivation of 
the stochastic differential equation (\ref{eq_2000}). This equation implies that for any stopping time $\tau$,
\[
d A_{t\wedge \tau} =  \mathbbm 1_{\{t<\tau\}} \cdot \left(
\sum_{i=1}^n H_{i,t} d B^i_t - A_t^2 \, dt \right).
\]
Write $\RR^{n \times n}_{symm}$ for the linear space of all symmetric
$n \times n$ matrices, equipped with the scalar product $\langle A, B \rangle = \tr[AB]$. Using It\^o's formula, we see that for any
$\mathcal C^2$-smooth function $F: \RR^{n \times n}_{symm} \rightarrow \RR$,
\begin{equation}\label{eq_finvdmsvni}
\begin{split}
d F(A_{t\wedge\tau})
& = \mathbbm 1_{\{t<\tau\}} \cdot  \sum_{i=1}^n \Tr ( \nabla F ( A_t ) H_{i,t} ) d B_{i,t}  \\
& + \mathbbm 1_{\{t<\tau\}} \left(
\frac12 \sum_{i=1}^n \nabla^2 F(A_t) ( H_{i,t}, H_{i,t} )
- \Tr ( \nabla F (A_t) A_t^2 )  \right) \, dt.
\end{split}
\end{equation}
Since $\mu$ is compactly-supported, and since $\mu_t$ has the same support
as $\mu$, the matrix-valued processes $(A_t)$ and $(H_{i,t})$ are uniformly bounded. 
Consequently, the local martingale part of the right-hand side of~\eqref{eq_finvdmsvni} is a genuine martingale, and the
absolutely-continuous part is integrable. By taking expectation
we thus get
\begin{equation}\label{eq_vfdlmhhguelti}
\frac{d}{dt} \E F(A_{t\wedge\tau})
= \E \left[ \mathbbm 1_{\{t<\tau\}} \left( \frac12 \sum_{i=1}^n  \nabla^2 F(A_t) ( H_{i,t}, H_{i,t} ) -  \Tr ( \nabla F (A_t) A_t^2 ) \right)  \right] .
\end{equation}
Consider the particular case where $F: \RR^{n \times n}_{symm} \rightarrow \RR$ takes the form
\[
F(A) = \Tr  f(A)
\]
for some smooth function $f\colon \R \to \R$. In this case, the gradient and Hessian 
of $F$  may be described explicitly.
Indeed, by the Hadamard perturbation lemma, for any symmetric matrix
$A = \sum_{i=1}^n \lambda_i \, u_i\otimes u_i \in \RR^{n \times n}$,
\begin{equation}\label{eq_rlrfrgghh}
\nabla F(A) = f'(A) = \sum_{i=1}^n f'(\lambda_i) u_i \otimes u_i.
\end{equation}
The corresponding formula 
for the Hessian of $F$ is sometimes called the Daleckii-Krein formula (e.g. \cite[Chapter V]{bhatia}). It  states that for any symmetric matrix $H \in \RR^{n \times n}$,
\begin{equation}\label{eq_ferjiferld}
\nabla^2 F ( A ) ( H , H )
= \sum_{i,j=1}^n \frac{ f'(\lambda_i) - f'(\lambda_j)}{\lambda_i - \lambda_j} \cdot \langle Hu_i,u_j\rangle^2 ,
\end{equation}
where the quotient is interpreted by continuity as $f''(\lambda_i)$ when $\lambda_i = \lambda_j$.
Recall the definitions (\ref{eq_4700}) and (\ref{eq_4701})
of $H_{i,t}$ and $\xi_{ij}(t)$ and observe that for any fixed $i,j \leq n$ we have
\[
\sum_{k=1}^n \langle H_{k,t} u_i(t) , u_j(t) \rangle^2 = \sum_{k=1}^n \vert \xi_{ijk} (t) \vert^2= \vert \xi_{ij} (t) \vert^2.
\]
Combining this with (\ref{eq_vfdlmhhguelti}), (\ref{eq_rlrfrgghh}), and (\ref{eq_ferjiferld}) yields the result.
\end{proof}
We will apply Lemma \ref{lem_1413} for a $\mathcal C^2$-smooth function
$f: [0, \infty) \rightarrow [0,\infty)$ satisfying
the following three conditions:
\begin{equation}\label{condit}
\begin{cases}
f \text{ is increasing} \\
f(x) = x^2 , & \forall x \geq r \\
f''(x) \leq D^2 f(x) , & \forall x \geq 0 
\end{cases}
\end{equation}
where $r,D > 0$ are parameters.
\begin{lemma} \label{lem_key}
Let $f: [0, \infty) \rightarrow [0,\infty)$ be a $\mathcal C^2$-function satisfying~\eqref{condit} with parameters $D>1$ and $r\in[2,3]$. Then for any stopping time $\tau$ and any fixed $t > 0$,
\begin{equation} \frac{d}{dt} \E \tr f(A_{t\wedge \tau}) \leq C \left( \frac 1t + \frac{D^2}{\sqrt{t}} \right) \cdot \E \tr f(A_{t\wedge\tau}),
\label{eq_1414}
\end{equation}
where $C > 0$ is a universal constant.
\end{lemma}

\begin{proof} Apply Lemma \ref{lem_1413}.
The second summand on the right-hand side of (\ref{eq_1016}) is non-positive and may therefore be ignored. 
To prove the lemma, it is enough to show that for any fixed $t > 0$, almost surely
\begin{equation}\label{eq_goal}
\sum_{i,j=1}^n
|\xi_{ij}(t)|^2 \frac{f'(\lambda_i (t)) - f'(\lambda_j(t))}
{\lambda_i (t) - \lambda_j (t)}
\leq C \left( \frac 1t + \frac{D^2}{\sqrt{t}} \right) \cdot \sum_{i=1}^n
f(\lambda_i(t)).
\end{equation}
Indeed, since $f$ is non-negative, for all $t > 0$ and $1 \leq i \leq n$, almost surely,
\[
f ( \lambda_i (t) ) \mathbbm 1_{\{t<\tau\}} \leq f( \lambda_i (t\wedge \tau) ). 
\]
Therefore, by multiplying~\eqref{eq_goal}
by $\mathbbm 1_{\{t < \tau\}}$ and taking expectation we obtained the desired inequality (\ref{eq_1414}).

\medskip Inequality~\eqref{eq_goal} is established in \cite{guan},
but for completeness we recall the argument. We omit the dependence
in $t$ in order to lighten notation.
Observe  that since $f'(x) = 2x$ when $x \geq r$, we have
\[
\begin{split}
\sum_{i,j=1}^n \frac{ f'(\lambda_i)-f'(\lambda_j) }{ \lambda_i -\lambda_j }
|\xi_{ij}|^2 \mathbbm  1_{\{ \min ( \lambda_i, \lambda_j ) \geq r   \} }
& = 2 \sum_{i,j} \vert \xi_{ij}\vert^2
 \mathbbm  1_{\{ \min ( \lambda_i , \lambda_j ) \geq r \} } \\
 & \leq 2 \sum_{i,j} \vert \xi_{ij}\vert^2 \mathbbm 1_{\{\lambda_i \geq r \}} .
\end{split}
\]
Additionally, 
\[
\begin{split}
\sum_{i,j=1}^n \frac{ f'(\lambda_i)-f'(\lambda_j) }{ \lambda_i -\lambda_j }
|\xi_{ij}|^2 \mathbbm  1_{\{ \lambda_i \geq r+1\}} \mathbbm 1_{\{ \lambda_j \leq r \} }
& \leq \sum_{i,j} \frac{ 2\lambda_i  }{ \lambda_i - \lambda_j }  \vert \xi_{ij} \vert^2
\mathbbm  1_{\{ \lambda_i \geq r+1\}} \mathbbm 1_{\{ \lambda_j \leq r \} }  \\
& \leq 8 \sum_{i, j} \vert\xi_{ij} \vert^2  \mathbbm 1_{\{ \lambda_i \geq r \} }.
\end{split}
\]
Here we  used the fact that $\lambda_i/(\lambda_i-\lambda_j) \leq r+1\leq 4$ when $\lambda_j \leq r$ and $\lambda_i \geq r+1$.
Moreover, since  $\xi_{ijk}$ is symmetric in $i,j$ and $k$, 
\[
\begin{split}
\sum_{i,j=1}^n \vert \xi_{ij} \vert^2  \mathbbm  1_{\{ \lambda_i \geq r \} }
& =  \sum_{i,j,k}  \xi_{ijk}^2  \mathbbm  1_{\{ \lambda_i \geq r \} } \\
& \leq 3   \sum_{i,j,k}  \xi_{ijk}^2  \mathbbm  1_{\{ \lambda_i \geq r \} } \mathbbm 1_{\{\max(\lambda_j , \lambda_k) \leq \lambda_i \}} \\
& \leq \frac{12}{\sqrt t} \sum_{i} \lambda_i^{5/2}
\mathbbm  1_{\{ \lambda_i \geq r \} } \\
& \leq \frac{12}{t} \sum_i \lambda_i^2  \mathbbm  1_{\{ \lambda_i \geq r \} } \\
& \leq \frac {12}t \sum_i f( \lambda_i ) .
\end{split}
\]
Here we used Lemma~\ref{lem_guan} (with $u=\lambda_i$)
and the fact that $\lambda_i \leq t^{-1}$. The application of 
Lemma~\ref{lem_guan} is legitimate since the probability measure $\mu_t$ is $t$-uniformly log-concave.
To summarize, thus far we have shown that the contribution to the left-hand side of \eqref{eq_goal}
of the indices $i,j$ for which
either both $\lambda_i$ and $\lambda_j$ are larger than $r$, or
else one of the two is less than $r$ and the other larger that $r+1$, is at most
$$ C \sum_{i, j} \vert\xi_{ij} \vert^2  \mathbbm 1_{\{ \lambda_i \geq r \} } \leq \frac{\tilde{C}}{t} \sum_{i=1}^n f(\lambda_i). $$
All other pairs of indices $i,j$ satisfy $\max (\lambda_i, \lambda_j ) \leq r+1$. By symmetry, it suffices to bound the contribution to the left-hand side of (\ref{eq_goal}) 
of all $i,j$ for which $\lambda_j \leq \lambda_i \leq r+1$.
Using~\eqref{condit} and
the fact that $f$ is increasing, we obtain
\begin{equation}\label{eqtrrr}
\begin{split}
\sum_{i,j=1}^n & \frac{ f'(\lambda_i)-f'(\lambda_j) }{ \lambda_i -\lambda_j }
|\xi_{ij}|^2 \mathbbm  1_{ \{  \lambda_j \leq \lambda_i \leq r+1 \} } \\
& \leq D^2 \sum_{i,j}  f(\lambda_i) \vert \xi_{ij} \vert^2 \mathbbm  1_{ \{ \lambda_j \leq \lambda_i \leq r+1 \} }   \\
& \leq  D^2 \sum_{i,j,k} f(\lambda_i) \xi_{ijk}^2 \mathbbm  1_{ \{ \max(\lambda_i,\lambda_j,\lambda_k) \leq r+1 \} } \\
& \qquad + D^2 \sum_{i,j,k} f(\lambda_i) \xi_{ijk}^2 \mathbbm  1_{ \{ \max(\lambda_i,\lambda_j)\leq r+1 \leq \lambda_k \} }.
\end{split}
\end{equation}
By Lemma~\ref{lem_guan} applied with $u = r+1$, and recalling that $r \leq 3$,
\[
\begin{split}
\sum_{i,j,k} f(\lambda_i) \xi_{ijk}^2
\mathbbm  1_{ \{ \max(\lambda_i,\lambda_j,\lambda_k) \leq r+1 \} }
& \leq \frac{4}{\sqrt{t}}
\sum_{i} f(\lambda_i)  \cdot (r+1)^{3/2} \lambda_i  \cdot \mathbbm  1_{ \{ \lambda_i \leq r+1 \} }  \\
& \leq \frac{C}{\sqrt t} \sum_{i=1}^n f ( \lambda_i ) .
\end{split}
\]
In order to bound the second term on the right-hand side of~\eqref{eqtrrr} we 
use that $f(\lambda_i) \leq f(r+1) = (r+1)^2\leq 16$ and then apply Lemma~\ref{lem_guan}
with $u = r+1$. We get
\[
\begin{split}
\sum_{i,j,k} f(\lambda_i ) \xi_{ijk}^2 \mathbbm 1_{\{\max(\lambda_i,\lambda_j)\leq r+1 \leq \lambda_k \} }
&\leq \frac{C'}{\sqrt{t}}  \sum_k \lambda_k
\mathbbm  1_{ \{ \lambda_k \geq r+1 \} } \\
& \leq \frac{C'}{\sqrt{t}}  \sum_k \lambda_k^2
\mathbbm  1_{ \{ \lambda_k \geq r+1 \} } \\
& \leq \frac{C'}{\sqrt{t}}  \sum_{k=1}^n f(\lambda_k) .
\end{split}
\]
This finishes the proof of (\ref{eq_goal}). 
\end{proof}
Lemma \ref{lem_key} is more flexible than Lemma 8 from Chen \cite{chen}, where functions of the form $f(t) = t^q$ are considered. Similarly to Guan \cite{guan}, we will apply 
Lemma \ref{lem_key} for the following family of functions. For
$D > 1$ and $r \in [2,3]$ we let $f_{D,r}: [0, \infty) \rightarrow \RR$ be a $\mathcal C^2$-smooth, positive, increasing function such that
\begin{equation}\label{eq_defffff}
f(x) = f_{D,r} (x) = \begin{cases}
 \e^{D(x - r)} & x \leq r- D^{-1} \\
 x^2  & x \geq r ,
 \end{cases}
\end{equation}
and
\begin{equation} \label{eqderiv}
f'' (x) \leq (12 D)^2 \cdot f(x) , \qquad \forall x \geq 0 .
\end{equation}

\begin{lemma} For any $D > 1$ and $r \in [2,3]$ there exists a $\mathcal C^2$-smooth, positive, increasing function $f = f_{D,r}: [0, \infty) \rightarrow \RR$ satisfying 
(\ref{eq_defffff}) and (\ref{eqderiv}). \label{lem_1444}
\end{lemma}

\begin{proof} Set $r_0 = r - D^{-1}$ and $L = 40$. We claim that there exists a positive, $\mathcal C^1$-smooth, $L$-Lipschitz function $h: [0, 1] \rightarrow \RR$
such that 
\begin{equation}  h(0) = 1/e, h(1) = 2 r/D, h'(0) = 1/e, h'(1) = 2 / D^2  \label{eq_2250} 
\end{equation}
and 
$$\int_{0}^1 h(t) dt =  r^2 - 1/e. 
$$
Once we find such a function $h$, we define for $x \in [r_0, r]$,
$$ f(x) = 1/e + D \int_{0}^{x-r_0} h \left( D t \right) dt, $$
while for $x \not \in [r_0, r]$ we define $f(x)$ according to (\ref{eq_defffff}). Observe that $f$ is a positive, increasing, $\mathcal C^2$-function
satisfying (\ref{eq_defffff}). The function $f$ clearly satisfies (\ref{eqderiv}) for all $x \not \in [r_0, r]$, 
while for $x \in [r_0, r]$, it satisfies (\ref{eqderiv}) since
 $$ f''(x)  \leq D^2 L \leq 120 D^2 / e \leq 120 D^2 f(x). $$
We still need to find a function $h$ satisfying the above properties. Write $\cH$ for the collection of all $\mathcal C^1$-smooth, positive, $L$-Lipschitz functions $h$ satisfying (\ref{eq_2250}).
Then $\cH$ is a convex set, and hence the range of the map $\cH \ni h \mapsto \int_0^1 h$ is an interval.
It thus suffices to find $h_0, h_1 \in \cH$ with 
\begin{equation}  \int_0^1 h_0 < r^2 - 1/e < \int_0^1 h_1. \label{eq_632} \end{equation}
In fact, by approximation, it suffices to find non-negative $L$-Lipschitz functions $h_0$ and $h_1$ 
satisfying (\ref{eq_632}) with $h_i(0) = 1/e$ and $h_i(1) = 2r/D$ for $i=0,1$. Recall that $L = 40, D > 1$ and $r \in [2,3]$.
The construction of $h_0$ and $h_1$ is now an elementary exercise.
\end{proof}

We are now in a position to prove Proposition \ref{prop_vnis}.
\begin{proof}[Proof of Proposition~\ref{prop_vnis}]
It suffices to treat the case $t < c$, where $c > 0$ is a universal constant. Fix $t\leq 2^{-8}$,
and for an integer $k \geq 0$ denote 
 $$ t_k = 2^{-8k} t. $$ 
 For $k \geq 0$ set
$$ D_k = t_k^{-1/4}, $$ and note 
 that $D_k >1$. 
Define a sequence $(r_k)_{k \geq 0}$ by
\begin{equation}\label{eq_defri}
r_0 = 3  , \qquad r_{k+1} = r_{k} - t_k^{1/8} , \; k \geq 0 .
\end{equation}
Since $t\leq 2^{-8}$,
\[
\sum_{k=0}^{\infty} t_k^{1/8} = \sum_{k=0}^{\infty} 2^{-k} t^{1/8} = 2 t^{1/8} \leq 1.
\]
From (\ref{eq_defri}) we thus see that $r_k \in [2,3]$ for all $k \geq 0$. 
Consider the function $f_k = f_{D_k,r_k}$ provided by Lemma \ref{lem_1444}.
Apply Lemma~\ref{lem_key} for the function $f_{k}$ and $D = 12 D_k$. 
Observe that for $s \in [t_{k+1},t_k]$ we have
$s \leq t_k = D_k^{-4}$ and hence 
\[
\frac{D_k^2}{\sqrt s} \leq \frac 1s .
\]
Lemma~\ref{lem_key} thus shows that
\[
\frac{d}{ds} \E  \Tr f_k ( A_{s\wedge \tau} )
\leq \frac {C_1} s \E \Tr f_k ( A_{s\wedge \tau} ) ,
\quad  \forall  s \in [t_{k+1},t_k] ,
\]
where $C_1 > 0$ is a universal constant.
Integrating this differential inequality yields
\begin{equation}\label{eq_step000}
\E \Tr f_{k} ( A_{t_{k}\wedge \tau} ) \leq  \left( \frac{t_k}{t_{k+1}} \right)^{C_1}
\E \Tr f_{k} (A_{t_{k+1} \wedge \tau} )  .
\end{equation}
Consider the function $g_k$ defined by
\[
g_k (x) = x^2 \mathbbm 1_{\{x \geq r_k \} }.
\]
Since $f_k = f_{D_k, r_k}$, we see from (\ref{eq_defffff}) that 
\begin{equation}\label{eq_tttoyoy}
g_k \leq f_k.
\end{equation}
We claim that 
\begin{equation}\label{eq_yonbe}
f_k \leq \frac 94  g_{k+1} + \exp(-t_{k}^{-1/8} ).
\end{equation}
Let us prove (\ref{eq_yonbe}). Since $t_k \leq 1$ and $D_k = t_k^{-1/4}$ we have
\[
r_{k+1} = r_{k} - t_k^{1/8} \leq r_k - t_k^{1/4} = r_k - D_k^{-1} .
\]
Therefore, if $x \leq r_{k+1}$ then by (\ref{eq_defffff}),
\[
f_{k} (x) \leq f_k (r_{k+1} ) = \exp (D_k (r_{k+1}-r_k))
= \exp(- t_k^{-1/4} t_k^{1/8} ) = \exp( -t_k^{-1/8} ).
\]
Hence (\ref{eq_yonbe}) holds true when we evaluate $f_k$ and $g_{k+1}$  at a point $x \in [0, r_{k+1}]$. 
If $x \geq r_{k}$ then $f_k(x) = x^2 = g_{k+1}(x)$
and (\ref{eq_yonbe}) holds true in this case too. Finally,  if $x \in [r_{k+1} , r_k]$ then
\[
f_k (x) \leq f_k (r_k) = r_k^2 \leq \frac{ r_k^2}{ r_{k+1}^2 } x^2 \leq \frac 94 x^2 = \frac 94 g_{k+1} (x),
\]
since $r_{k}, r_{k+1} \in [2,3]$. We have thus completed the proof of (\ref{eq_yonbe}).
By substituting ~\eqref{eq_tttoyoy} and \eqref{eq_yonbe}  into~\eqref{eq_step000}
and setting
\[
F_k = \E \Tr g_k ( A_{t_k \wedge \tau} )
\]
we  obtain
\[
F_k \leq \left( \frac{t_k}{t_{k+1}} \right)^{C_1} \left( \frac 94 F_{k+1} + n \exp ( - t_{k}^{-1/8} ) \right) .
\]
Since $t_{k}/t_{k+1} = 2^8$ and $9/4 \leq 2^2$
this inequality implies that
\begin{equation} 
F_k \leq 2^{C_2} \left( F_{k+1} + n \exp ( - t_{k}^{-1/8} ) \right) ,
\label{eq_715} \end{equation}
where $C_2 = 8C_1 + 2$. From the recursive inequality (\ref{eq_715}) we obtain that for any $k \geq 1$,
\begin{equation}\label{eq_stepPPPP}
 F_0  \leq 2^{C_2 k} F_k + n
\cdot \sum_{i=0}^{k-1}
2^{C_2 (i+1)} \exp ( - t_{i}^{-1/8}).
\end{equation}
Observe that
\[
t_i^{-1/8}  = 2^i t^{-1/8}
= (2^i-1)t^{-1/8} + t^{-1/8} \geq 2 (2^i-1) + t^{-1/8}.
\]
We thus conclude from (\ref{eq_stepPPPP}) 
and from the inequality $2^{C_2 k} \leq t_k^{-C_2}$ that for $k \geq 1$,
\begin{equation}  F_0 \leq t_k^{-C_2} F_k + n e^{-t^{-1/8}} 
\cdot \sum_{i=0}^{k-1}
2^{C_2 (i+1)} e^{ - 2 (2^i-1) }
\leq t_k^{-C_2} F_k + C_3 n \cdot e^{-t^{-1/8}},
 \label{eq_5666} \end{equation}
where the last passage follows from the fact that the series is clearly convergent. 
Next, we claim that  $t_k^{-C_2} F_k$ tends to $0$
when $k$ tends to $+\infty$.
Indeed, recall that $r_k \geq 2 $ for all $k$. Therefore,
\[
\begin{split}
F_k = \E \Tr g_k ( A _{t_k \wedge \tau} )
& \leq \E \left[ \sum_{i=1}^n \lambda_i ( t_k \wedge \tau )^2
\mathbbm 1_{\{ \lambda_i ( t_k \wedge \tau ) \geq 2 \} } \right]  \\
& \leq  \E \left[ \vert A_{t_k\wedge \tau} \vert^2
\mathbbm 1_{\{ \Vert A_{t_k\wedge \tau} \Vert_{op} \geq 2 \} }\right].
\end{split}
\]
Thus  it suffices to prove that when $t \to 0$,
\begin{equation}\label{eq_enough}
 \E \left[ \vert A_{t\wedge \tau} \vert^2
\mathbbm 1_{ \{ \Vert A_{t\wedge \tau} \Vert_{op} \geq 2 \} } \right] = o(t^{C_2}).
\end{equation}
Recall from (\ref{eq_727}) that $\vert A_t \vert \leq \tilde{C}_{\mu}$
almost surely, for some constant $\tilde{C}_{\mu}$ depending only on the 
compactly-supported measure $\mu$. It follows that there exists a constant $D_\mu > 0$ depending only on
$\mu$ such that
\begin{equation} 
\E \left[ \vert A_{t\wedge \tau} \vert^2
\mathbbm 1_{ \{ \Vert A_{t\wedge \tau} \Vert_{op} \geq 2 \} } \right]
 \leq D_\mu \cdot \prob ( \Vert A_{t\wedge \tau} \Vert_{op} \geq 2 )
\leq D_\mu \cdot \prob ( \tau_* \leq t ),
\label{eq_730} \end{equation}
where $\tau_*$ was defined in (\ref{eq_deftaustar}).
Inequality (\ref{eq_730}) combined with the qualitative estimate (\ref{eq_1738})
imply~(\ref{eq_enough}), which proves the claim. Consequently, we 
may let $k$ tend to $+\infty$ in (\ref{eq_5666}), and obtain
$$ F_0 \leq C_3  n \cdot \exp ( - t^{-1/8} ). $$
By using the inequality
\[
\mathbbm 1_{\{ x\geq 3 \} } \leq x^2 \mathbbm 1_{\{x\geq 3\}}
= g_0 (x)  ,
\]
we finally obtain
\[
\sum_{i=1}^n \prob( \lambda_i ( t\wedge \tau ) \geq 3 )
\leq  \E \Tr g_0 ( A_{t\wedge \tau} )  = F_0 \leq C_3 n \cdot
\exp(-t^{-1/8} )   ,
\]
and the proof is complete. 
\end{proof}

\section{Proofs of the main results}
\label{sec6}

We continue with the notation and assumptions of Section \ref{sec4}.
Thus $\mu$ is an isotropic, compactly-supported, log-concave probability measure in $\RR^n$.
Recall the covariance process $(A_t)_{t \geq 0}$ 
and its eigenvalues $\lambda_1 (t) \geq \dotsb \geq \lambda_n (t)>0$.  For  $k= 1, \dotsc,n$ we consider 
the stopping time
$$ 
\tau_k = \inf \{ t >0 \, ; \,  \lambda_{k}(t) \geq 3 \} .
$$
Proposition \ref{prop_vnis} admits the following corollary:

\begin{corollary} \label{cor_main}
For $k=1,\ldots,n$ and $t > 0$,
\begin{equation}\label{eq_gaaaa}
\prob ( \tau_k  \leq t ) \leq C_1 \frac nk \cdot \exp ( - t^{-\alpha} ) ,
\end{equation}  
and
\begin{equation} 
\E \tau_k^{-2} \leq C_2 \left( 1 + \log \frac nk \right)^\beta.
\label{eq_1544} \end{equation}
Here, $C_1,C_2, \alpha, \beta > 0$ are universal constants (in fact $\alpha = 1/8$ and $\beta = 16$).
\end{corollary}
%
%
\begin{proof}
Fix $1 \leq k \leq n$, and apply Proposition~\ref{prop_vnis} for the stopping
time $\tau_k$ to obtain
\begin{equation}\label{eq_nveimi}
\sum_{i=1}^n \prob ( \lambda_i ( t\wedge \tau_k ) \geq 3 )
\leq C n \cdot \exp ( - t^{-1/8} ) .
\end{equation}
Observe that if $\tau_k \leq t$ then at time $t\wedge \tau_k = \tau_k$ the $k$
largest eigenvalues of $A_{t\wedge \tau_k}$ are greater than or equal to $3$.
This implies that
\[
\prob ( \tau_k \leq t ) \leq \prob ( \lambda_i ( t\wedge \tau_k ) \geq 3 ) , \qquad \forall i \leq k.
\]
Consequently,
\begin{equation}
k \cdot \prob ( \tau_k \leq t ) \leq \sum_{i=1}^k \prob ( \lambda_i ( t\wedge \tau_k ) \geq 3 ) \leq
\sum_{i=1}^n \prob ( \lambda_i ( t\wedge \tau_k ) \geq 3 ) .
\label{eq_738} \end{equation}
From  \eqref{eq_nveimi} and (\ref{eq_738}) we deduce ~\eqref{eq_gaaaa}. In order 
to prove (\ref{eq_1544}) we set 
\[
x_0 = 2^{1/\alpha} \left( \log  \frac nk \right)^{1/\alpha},
\]
for $\alpha = 1/8$.
By ~\eqref{eq_gaaaa}, 
\[
\begin{split}
\E \tau_k^{-2} & = \int_0^\infty 2x \cdot \prob ( \tau_k \leq x^{-1} ) \, dx \\
& \leq x_0^2 + C_1 \frac{n}{k} \int_{x_0}^\infty 2x \e^{-x^{\alpha} } \, dx  \\
& \leq x_0^2 + C_1 \cdot \left( \frac{n}{k} \e^{- \frac 12 x_0^{\alpha} } \right) \int_0^\infty 2x
\e^{- \frac 12 x^\alpha} \, dx \\
& = x_0^2 + C' ,
\end{split}
\]
which yields the desired result.
\end{proof}
We may now prove the following variant of Guan's estimate (\ref{eq_guan}).

\begin{theorem}\label{thm_guantype} 
Let $\mu$ be an isotropic, compactly-supported, log-concave probability measure in $\RR^n$. Then,
with the notation of Corollary \ref{cor_311},
\[
\E \left[
\sum_{i=1}^n \exp \left( 2 \int_0^1 \lambda_{i} (t) \, dt \right) \right]
\leq C n,
\]
where $C > 0$ is a universal constant.
\end{theorem}

\begin{proof}
Recall from (\ref{eq_1817}) that $A_t \leq t^{-1} \cdot \id$ for all $t > 0$, almost surely. Therefore $\lambda_{k}(t) \leq t^{-1}$ for
all $k$ and  $t$. Consequently, for $k=1,\ldots,n$, 
\[
\int_0^{1} \lambda_{k} (t) \, dt \leq 3 (\tau_k\wedge 1) + \int_{\tau_k \wedge 1}^1 \frac{dt}{t}
\leq 3 - \log ( \tau_k \wedge 1) .
\]
Hence,
\[
\exp \left( 2\int_0^{1} \lambda_{k}(t) \, dt \right) \leq \e^6 \cdot \left( \tau_k^{-2} \vee 1 \right)
\leq \e^6 \cdot \left( \tau_k^{-2} + 1 \right),
\]
where $a \vee b = \max \{a,b\}$. From Corollary \ref{cor_main} we thus obtain that 
\[
\E \left[ \sum_{i=1}^n \exp \left( 2 \int_0^1 \lambda_{i}(t) dt \right) \right]  \leq C' \sum_{k=1}^n \left(  1 + \log \frac nk \right)^\beta \leq \tilde{C} n,
\]
with $\beta = 16$, where the last passage follows from the fact that the monotone function $\left( 1+ \log \frac 1x \right)^\beta$ is integrable in the interval $[0,1]$, and consequently the Riemann sum
corresponding to this integral is bounded by a universal constant.
\end{proof}


As explained in the Introduction, Theorem \ref{thm1} 
is a consequence of Theorem~\ref{thm2}.

\begin{proof}[Proof of Theorem~\ref{thm2}]
In the case where $\mu$ is compactly-supported, 
Corollary \ref{cor_311} and Theorem \ref{thm_guantype} imply that
\begin{equation}\label{eq_jjjj}
\sum_{i=1}^n \Vert x_i\Vert_{H^{-1} (\mu)}^2 \leq Cn.
\end{equation}
We need to eliminate the assumption that $\mu$ is compactly-supported.
Consider the space $\mathcal X$ of all isotropic, log-concave probability measures on $\RR^n$,
equipped with the topology of weak convergence, i.e.,  the minimal topology under which
$\mu \mapsto \int \vphi d \mu$ is a continuous functional, for any continuous, compactly-supported 
function $\vphi: \RR^n \rightarrow \RR$. In particular, for any 
$\vphi \in \mathcal C_c^{\infty}(\RR^n)$ and $i=1,\ldots,n$, the functional 
\[
\mu \mapsto \int_{\R^n} [ 2 x_i \phi (x) -  \vert \nabla \phi (x)\vert^2  ] \, d \mu (x)
\]
is continuous in $\cX$. Observe that we may rewrite  (\ref{eq_1528}) as 
\[
\Vert x_i\Vert^2_{H^{-1} (\mu)} =
\sup_{\phi \in \mathcal C^\infty_c (\R^n)}
\left\{ \int_{\R^n} [ 2 x_i \phi (x) -  \vert \nabla \phi (x)\vert^2    ] \, d \mu (x) \right\}.
\]
This implies that $\mu \mapsto \Vert x_i\Vert^2_{H^{-1} (\mu)}$ is lower semi-continuous in $\mathcal X$.
Therefore (\ref{eq_jjjj}) holds true for any $\mu$ in the closure in $\cX$ of  the collection 
of compactly-supported measures. 

\medskip All that remains is to show that any isotropic, log-concave probability measure 
is the weak limit of a sequence of compactly-supported, isotropic, log-concave probability measures. 
This is a standard fact, which may be proved as follows: Let $\mu$ be an isotropic, log-concave probability measure in $\RR^n$. Write $\nu_k$ for the conditioning of $\mu$ on the ball of radius $k$ centered at the origin.
Write $b_k \in \RR^n$ for the barycenter of $\nu_k$ and set $$ T_k(x) = \cov(\nu_k)^{-1/2}(x - b_k), \qquad x \in \RR^n, k \geq 1. $$
Let $\mu_k$ be the push-forward of $\nu_k$ under $T_k$, which is a compactly-supported, isotropic, log-concave probability measure.
Clearly  $T_k(x) \longrightarrow  x$ for any $x \in \RR^n$, and the convergence is locally-uniform. Therefore $\mu_k \longrightarrow  \mu$
in the topology of weak convergence of measures. This completes the proof.
\end{proof}

\medskip
\noindent Department of Mathematics,
Weizmann Institute of Science,
Rehovot 76100, Israel. \\
{\it e-mail:} \verb"boaz.klartag@weizmann.ac.il"

\medskip
\noindent Universit\'e de Poitiers, CNRS, LMA, Poitiers, France. \\
{\it e-mail:} \verb"joseph.lehec@univ-poitiers.fr"


\begin{thebibliography}{10}

\bibitem{ABP}
M.~Anttila, K.~Ball, and I.~Perissinaki.
\newblock The central limit problem for convex bodies.
\newblock {\em Trans. Amer. Math. Soc.}, 355(12):4723--4735, 2003.

\bibitem{ball_perissinaki}
K.~Ball and I.~Perissinaki.
\newblock The subindependence of coordinate slabs in {$l^n_p$} balls.
\newblock {\em Israel J. Math.}, 107:289--299, 1998.

\bibitem{BCE}
F.~Barthe and D.~Cordero-Erausquin.
\newblock Invariances in variance estimates.
\newblock {\em Proc. Lond. Math. Soc. (3)}, 106(1):33--64, 2013.

\bibitem{barthe_klartag}
F.~Barthe and B.~Klartag.
\newblock Spectral gaps, symmetries and log-concave perturbations.
\newblock {\em Bull. Hellenic Math. Soc.}, 64:1--31, 2020.

\bibitem{BW}
F.~Barthe and P.~Wolff.
\newblock Volume properties of high-dimensional {O}rlicz balls.
\newblock In {\em High dimensional probability {IX}---the ethereal volume},
  volume~80 of {\em Progr. Probab.}, pages 75--95. Birkh\"auser/Springer, 2023.

\bibitem{bhatia}
R.~Bhatia.
\newblock {\em Matrix analysis}, volume 169 of {\em Graduate Texts in
  Mathematics}.
\newblock Springer, 1997.

\bibitem{BCG1}
S.~G. Bobkov, G.~Chistyakov, and F.~G\"{o}tze.
\newblock {\em Concentration and {G}aussian approximation for randomized sums},
  volume 104 of {\em Probability Theory and Stochastic Modelling}.
\newblock Springer, 2023.

\bibitem{BobKol}
S.~G. Bobkov and A.~Koldobsky.
\newblock On the central limit property of convex bodies.
\newblock In {\em Geometric aspects of functional analysis (2001--02)}, volume
  1807 of {\em Lecture Notes in Math.}, pages 44--52. Springer, 2003.

\bibitem{bourgain}
J.~Bourgain.
\newblock On the distribution of polynomials on high-dimensional convex sets.
\newblock In {\em Geometric aspects of functional analysis (1989--90)}, volume
  1469 of {\em Lecture Notes in Math.}, pages 127--137. Springer, 1991.

\bibitem{bourgain_milman}
J.~Bourgain and V.~D. Milman.
\newblock New volume ratio properties for convex symmetric bodies in {${\bf
  R}^n$}.
\newblock {\em Invent. Math.}, 88(2):319--340, 1987.

\bibitem{BGVV}
S.~Brazitikos, A.~Giannopoulos, P.~Valettas, and B.-H. Vritsiou.
\newblock {\em Geometry of isotropic convex bodies}, volume 196 of {\em
  Mathematical Surveys and Monographs}.
\newblock American Mathematical Society, 2014.

\bibitem{CM}
R.~H. Cameron and W.~T. Martin.
\newblock The transformation of {Wiener} integrals by nonlinear
  transformations.
\newblock {\em Trans. Am. Math. Soc.}, 66:253--283, 1949.

\bibitem{chen}
Y.~Chen.
\newblock An almost constant lower bound of the isoperimetric coefficient in
  the {KLS} conjecture.
\newblock {\em Geom. Funct. Anal. (GAFA)}, 31(1):34--61, 2021.

\bibitem{chigansky}
P.~Chigansky.
\newblock Introduction to nonlinear filtering, 2005.
\newblock Lecture notes for a course given at the Weizmann Institute of
  Science. Available at:
  \url{https://pluto.huji.ac.il/~pchiga/teaching/Filtering/filtering-v0.2.pdf}.

\bibitem{DFGZ}
B.~Dadoun, M.~Fradelizi, O.~Gu\'edon, and P.-A. Zitt.
\newblock Asymptotics of the inertia moments and the variance conjecture in
  {S}chatten balls.
\newblock {\em J. Funct. Anal.}, 284(2):109741, 2023.

\bibitem{diaconis_freedman}
P.~Diaconis and D.~Freedman.
\newblock Asymptotics of graphical projection pursuit.
\newblock {\em Ann. Statist.}, 12(3):793--815, 1984.

\bibitem{eldan1}
R.~Eldan.
\newblock Thin shell implies spectral gap up to polylog via a stochastic
  localization scheme.
\newblock {\em Geom. Funct. Anal. (GAFA)}, 23(2):532--569, 2013.

\bibitem{EK}
R.~Eldan and B.~Klartag.
\newblock Approximately {G}aussian marginals and the hyperplane conjecture.
\newblock In {\em Concentration, functional inequalities and isoperimetry},
  volume 545 of {\em Contemp. Math.}, pages 55--68. Amer. Math. Soc., 2011.

\bibitem{fleury}
B.~Fleury.
\newblock Concentration in a thin {E}uclidean shell for log-concave measures.
\newblock {\em J. Funct. Anal.}, 259(4):832--841, 2010.

\bibitem{guan}
Q.~Guan.
\newblock A note on {Bourgain}'s slicing problem.
\newblock Preprint, {arXiv}:2412.09075, 2024.

\bibitem{guedon_milman}
O.~Gu\'edon and E.~Milman.
\newblock Interpolating thin-shell and sharp large-deviation estimates for
  isotropic log-concave measures.
\newblock {\em Geom. Funct. Anal. (GAFA)}, 21(5):1043--1068, 2011.

\bibitem{hartman}
P.~Hartman.
\newblock {\em Ordinary differential equations}.
\newblock {J}ohn {W}iley and {S}ons, {Inc.}, {New {York}-{London}-{Sydney}},
  1964.

\bibitem{HJ}
R.~A. Horn and C.~R. Johnson.
\newblock {\em Matrix analysis}.
\newblock Cambridge University Press, second edition, 2013.

\bibitem{JLV}
A.~Jambulapati, Y.~T. Lee, and S.~S. Vempala.
\newblock A {Slightly} {Improved} {Bound} for the {KLS} {Constant}.
\newblock Preprint, {arXiv}:2208.11644, 2022.

\bibitem{K_quarter}
B.~Klartag.
\newblock On convex perturbations with a bounded isotropic constant.
\newblock {\em Geom. Funct. Anal. (GAFA)}, 16(6):1274--1290, 2006.

\bibitem{K_clt}
B.~Klartag.
\newblock A central limit theorem for convex sets.
\newblock {\em Invent. Math.}, 168(1):91--131, 2007.

\bibitem{K_poly}
B.~Klartag.
\newblock Power-law estimates for the central limit theorem for convex sets.
\newblock {\em J. Funct. Anal.}, 245(1):284--310, 2007.

\bibitem{K_ptrf}
B.~Klartag.
\newblock A {B}erry-{E}sseen type inequality for convex bodies with an
  unconditional basis.
\newblock {\em Probab. Theory Related Fields}, 145(1-2):1--33, 2009.

\bibitem{K_euro}
B.~Klartag.
\newblock High-dimensional distributions with convexity properties.
\newblock In {\em European {C}ongress of {M}athematics}, pages 401--417. Eur.
  Math. Soc., 2010.

\bibitem{K_root}
B.~Klartag.
\newblock Logarithmic bounds for isoperimetry and slices of convex sets.
\newblock {\em Ars Inveniendi Analytica}, Paper No. 4, 17pp, 2023.

\bibitem{KL}
B.~Klartag and J.~Lehec.
\newblock Bourgain's slicing problem and {KLS} isoperimetry up to polylog.
\newblock {\em Geom. Funct. Anal. (GAFA)}, 32(5):1134--1159, 2022.

\bibitem{KL_slicing}
B.~Klartag and J.~Lehec.
\newblock Affirmative resolution of {Bourgain}'s slicing problem using {Guan}'s
  bound.
\newblock Preprint, {arXiv}:2412.15044, 2024.

\bibitem{KLbull}
B.~Klartag and J.~Lehec.
\newblock Isoperimetric inequalities in high-dimensional convex sets.
\newblock Preprint, {arXiv}:2406.01324, 2024.
\newblock To appear in Bull. AMS.

\bibitem{KP}
B.~Klartag and E.~Putterman.
\newblock Spectral monotonicity under {Gaussian} convolution.
\newblock {\em Ann. Fac. Sci. Toulouse, Math. (6)}, 32(5):939--967, 2023.

\bibitem{KolesMil}
A.~V. Kolesnikov and E.~Milman.
\newblock The {KLS} isoperimetric conjecture for generalized {O}rlicz balls.
\newblock {\em Ann. Probab.}, 46(6):3578--3615, 2018.

\bibitem{legall}
J.-F. Le~Gall.
\newblock {\em Brownian Motion, Martingales and Stochastic Calculus}.
\newblock Springer, 2016.

\bibitem{LV_focs}
Y.~T. Lee and S.~Vempala.
\newblock Eldan's stochastic localization and the {KLS} hyperplane conjecture:
  an improved lower bound for expansion.
\newblock In {\em 58th {A}nnual {IEEE} {S}ymposium on {F}oundations of
  {C}omputer {S}cience---{FOCS} 2017}, pages 998--1007. IEEE Computer Soc., Los
  Alamitos, CA, 2017.

\bibitem{LRog}
T.~Lindvall and L.~C.~G. Rogers.
\newblock Coupling of multidimensional diffusions by reflection.
\newblock {\em Ann. Probab.}, 14:860--872, 1986.

\bibitem{nsv}
F.~Nazarov, M.~Sodin, and A.~Volberg.
\newblock The geometric {K}annan-{L}ov\'asz-{S}imonovits lemma, dimension-free
  estimates for the distribution of the values of polynomials, and the
  distribution of the zeros of random analytic functions.
\newblock {\em Algebra i Analiz}, 14(2):214--234, 2002.

\bibitem{paouris}
G.~Paouris.
\newblock Concentration of mass on convex bodies.
\newblock {\em Geom. Funct. Anal. (GAFA)}, 16(5):1021--1049, 2006.

\bibitem{pass}
B.~Pass.
\newblock Multi-marginal optimal transport: theory and applications.
\newblock {\em ESAIM, Math. Model. Numer. Anal.}, 49(6):1771--1790, 2015.

\bibitem{RV}
J.~Radke and B.-H. Vritsiou.
\newblock On the thin-shell conjecture for the {S}chatten classes.
\newblock {\em Ann. Inst. Henri Poincar\'e{} Probab. Stat.}, 56(1):87--119,
  2020.

\bibitem{roc}
R.~T. Rockafellar.
\newblock {\em Convex analysis}.
\newblock Princeton University Press, 1970.

\bibitem{sudakov}
V.~N. Sudakov.
\newblock Typical distributions of linear functionals in finite-dimensional
  spaces of high dimension.
\newblock {\em Dokl. Akad. Nauk SSSR}, 243(6):1402--1405, 1978.
\newblock English translation: Soviet Math. Dokl. 19 (1978), no. 6, 1578–1582
  (1979).

\bibitem{Vil}
C.~Villani.
\newblock {\em Topics in optimal transportation}, volume~58 of {\em Graduate
  Studies in Mathematics}.
\newblock American Mathematical Society, 2003.

\bibitem{will}
D.~Williams.
\newblock {\em Probability with martingales}.
\newblock Cambridge University Press, 1991.

\end{thebibliography}
\end{document}